\pdfoutput=1
\RequirePackage{ifpdf}
\ifpdf 
\documentclass[pdftex]{sigma}
\else
\documentclass{sigma}
\fi

\newcommand{\C}{{\mathbb C}}
\newcommand{\Z}{{\mathbb Z}}

\newcommand{\Y}{{\mathbb Y}}
\newcommand{\B}{{\mathcal B}}
\newcommand{\F}{{\mathcal F}}

\newcommand{\N}{{\mathcal N}}

\newcommand{\bla}{\boldsymbol{\lambda}}
\newcommand{\bmu}{\boldsymbol{\mu}}
\newcommand{\bnu}{\boldsymbol{\nu}}

\numberwithin{equation}{section}

\newtheorem{Theorem}{Theorem}[section]
\newtheorem{Corollary}[Theorem]{Corollary}
\newtheorem{Lemma}[Theorem]{Lemma}
\newtheorem{Proposition}[Theorem]{Proposition}
{ \theoremstyle{definition}
\newtheorem{Definition}[Theorem]{Definition}
\newtheorem{Remark}[Theorem]{Remark} }

\begin{document}

\newcommand{\arXivNumber}{2004.13916}

\renewcommand{\PaperNumber}{050}

\FirstPageHeading

\ShortArticleName{On $q$-Isomonodromic Deformations and $q$-Nekrasov Functions}

\ArticleName{On $\boldsymbol{q}$-Isomonodromic Deformations\\ and $\boldsymbol{q}$-Nekrasov Functions}

\Author{Hajime NAGOYA}

\AuthorNameForHeading{H.~Nagoya}

\Address{School of Mathematics and Physics, Kanazawa University, Kanazawa, Ishikawa 920-1192, Japan}
\Email{\href{mailto:nagoya@se.kanazawa-u.ac.jp}{nagoya@se.kanazawa-u.ac.jp}}

\ArticleDates{Received June 02, 2020, in final form May 04, 2021; Published online May 13, 2021}

\Abstract{We~construct a fundamental system of a $q$-difference Lax pair of rank $N$ in~terms of 5d Nekrasov functions with $q=t$. Our fundamental system degenerates by the limit $q\to 1$ to a fundamental system of a differential Lax pair, which yields the Fuji--Suzuki--Tsuda system. We~introduce tau functions of our system as Fourier transforms of~5d Nekrasov functions. Using asymptotic expansions of the fundamental system at~$0$ and~$\infty$, we obtain several determinantal identities of the tau functions.}


\Classification{39A13; 33E17; 05A30}\vspace{-1ex}

\section{Introduction}

Isomonodromic deformations of linear differential equations~\cite{Jimbo-Miwa-Ueno} have been studied intensively and admit many applications both in mathematics and physics. The recent paper~\cite{GIL} revealed a~connection between isomonodromic deformation and the Liouville conformal field theory that the tau function of the sixth Painlev\'e equation $\mathrm{P_{VI}}$ is a Fourier transform of the Virasoro conformal block with central charge $c=1$. Later, the tau functions of other Painlev\'e equations $\mathrm{P_V}$, $\mathrm{P_{IV}}$, $\mathrm{P_{III}}$, $\mathrm{P_{II}}$ were also recognized as Fourier transforms of irregular conformal blocks~\cite{GIL1,Lisovyy-Nagoya-Roussillon, Nagoya2015, Nagoya2018}. Furthermore, the tau functions of isomonodromic deformation of particular Fuchsian systems of rank $N$ was constructed by semi-degenerate conformal blocks of~$W_N$-algebra with central charge $c=N-1$~\cite{GIL2}.

By AGT correspondence~\cite{AGT}, Virasoro conformal blocks correspond to Nekrasov partition functions of 4d supersymmetric gauge theories. Hence, the tau functions of the Painlev\'e equations have explicit series expansions at regular singular points. Correspondence between partition functions of gauge theories and tau functions of Painlev\'e equations was investigated in~\cite{BLMST}. In~\cite{BS2}, using Nakajima--Yoshioka blow-up equations~\cite{Nakajima-Yoshioka} for the pure Nekrasov partition functions, it was proved that the tau function of the Painlev\'e $\mathrm{III}_3$ ($\mathrm{P_{III_3}}$) in terms of the pure Nekrasov partition function satisfies the bilinear equation for it.

One of the reasons why conformal blocks appear in the construction of isomonodromic tau functions is that when the central charge takes particular values, monodromy invariant finite-systems can be constructed from conformal blocks with the degenerate field~\cite{GIL2,Iorgov-Lisovyy-Teschner}. We~note that by definition, local monodromies of conformal blocks are independent of their singular points, and connection problem of semi-degenerate $W_N$-conformal blocks with the degenerate field is reduced to connection problem of the Clausen--Thomae hypergeometric function ${}_{N}F_{N-1}$. Recall that $W_2$ is the Virasoro algebra.

A $q$-analog of isomonodromic deformation of linear differential equations can be considered. It corresponds to study connection preserved deformation of linear $q$-difference equations, that is, $q$-difference Lax pairs with rational coefficients of the independent variable. In~\cite{JS}, the $q$-Painlev\'e VI ($q$-$\mathrm{P_{VI}}$) was obtained from the compatibility condition of a Lax pair of rank~2. In~\cite{Jimbo-Nagoya-Sakai}, fundamental solutions of the Lax pair were constructed in terms of the $q$-conformal blocks given as expectation values of the intertwining operators of the Ding--Iohara--Miki alge\-bra~\cite{AFS}, and expressed in terms of the $q$-Nekrasov functions. As a result, a general solution to the $q$-Pain\-lev\'e~VI is expressed explicitly, and moreover, it was conjectured that tau functions of~\mbox{$q$-$\mathrm{P_{VI}}$} defined as Fourier transforms of the $q$-conformal blocks satisfy bilinear equations. Later, by taking degeneration of~$q$-$\mathrm{P_{VI}}$, explicit series representations for the tau functions of the $q$-Painlev\'e~V, $\mathrm{III_1}$, $\mathrm{III_2}$, and $\mathrm{III_3}$ equations were obtained~\cite{Matsuhira-Nagoya}. The series representations of the tau function for $q$-$\mathrm{P_{{III}_3}}$ expressed by the pure 5d Nekrasov function had been proposed in~\cite{BS1} for surface type~$A_7^{(1)\prime}$ and~\cite{BGM} for surface type $A_7^{(1)}$. The tau functions obtained in~\cite{Matsuhira-Nagoya} are equivalent to those proposed in~\cite{BGM,BS1}. A Fredholm determinant representation for the tau functions of~\mbox{$q$-$\mathrm{P_{III_3}}$} was presented in~\cite{BGT}.

In~this paper, we construct a fundamental system of a Lax pair of rank $N$ in terms of 5d Nekrasov functions with $q=t$. Our fundamental system degenerates by the limit $q\to 1$ to
 a~fundamental system of a Lax pair $L_{M,N}$ ($M\in\mathbb{Z}_{\geq 1}$) describing isomonodromy deformation of~a~Fuchsian system
with $M+3$ regular singular points
\begin{gather*}
	\frac{{\rm d}Y}{{\rm d}z}=\sum_{i=0}^{M+1}\frac{A_i}{z-x_i}Y,
\end{gather*}
where $Y=Y(z)$, $A_i$ are $N$ by $N$ matrices, and
whose {\it spectral type} is given by the ($M+3$)-tuple
\begin{gather*}
\left(1^N\right),\quad \left(1^N\right),\quad \left(N-1,1\right),\quad\dots,\quad\left(N-1,1\right).
\end{gather*}
Here, the spectral type expresses multiplicities of the eigenvalues of the coefficient matrices $A_i$.
The corresponding isomonodromic system was studied from the point of view of UC hierarchy~\cite{Tsuda2010}. When $M=1$, the isomonodromic system is also derived from similarity condition of~the Drinfeld--Sokolov hierarchy~\cite{Fuji-Suzuki2009} ($N=3$),~\cite{Suzuki2010} ($N\geq 2$). We~call this isomonodromic system the Fuji--Suzuki--Tsuda system. Since a fundamental system of~$L_{M,N}$ was constructed by~$W_N$ conformal blocks~\cite{GIL2}, the uniqueness of a fundamental system of a Fuchsian system gives another justification of AGT correspondence. We~note that to be rigorous, we have to~prove convergence of~$W_N$ conformal blocks and 5d Nekrasov partition functions, and moreover, Fourier transforms of them.

A key ingredient to construct a monodromy invariant fundamental system from conformal blocks is the braiding relation interchanging places of the variable $x$ of the degenerate field with other singular points. In~differential cases, the braiding relation is derived from $W_N$ conformal field theory. In~$q$-difference cases; however, we have not yet understood braiding relations of~$q$-conformal blocks representation theoretically. Instead, taking advantage of explicit expressions of~$q$-Nekrasov functions, we directly prove contiguity relations of~$q$-Nekrasov functions with special values by generalizing a method used in~\cite{Jimbo-Nagoya-Sakai}. Consequently, we solve the connection problem of~$q$-conformal blocks with the degenerate field.

The remainder of this paper is organized as follows. In~Section~\ref{S2}, we introduce $q$-conformal blocks and investigate the connection problem of degenerate $q$-conformal blocks. In~Section~\ref{S3}, we~construct a fundamental system of a particular Lax pair of rank $N$, using degenerate $q$-con\-formal blocks. We~introduce tau functions of our system as Fourier transforms of~$q$-conformal blocks. Finally, we present determinantal identities of the tau functions.

\medskip

\noindent
{\bf Notation.} Throughout the paper we fix $q\in\C^\times$ such that $|q|<1$.
We~set
\begin{gather*}
[u]=\frac{1-q^u}{1-q},\qquad
(a;q)_n=\prod_{j=0}^{n-1}\big(1-aq^j\big),
\\[-1ex]
(a_1,\dots,a_k;q)_{\infty}=\prod_{j=1}^k(a_j;q)_\infty,\qquad
(a;q,q)_\infty=\prod_{j,k=0}^\infty\big(1-aq^{j+k}\big).
\end{gather*}
We~use the $q$-gamma function, $q$-Barnes function and the $q$-theta function defined by
\begin{gather*}
\Gamma_q(u)=\frac{(q;q)_\infty}{(q^u;q)_\infty}(1-q)^{1-u},\qquad
G_q(u)=\frac{(q^u;q,q)_\infty}{(q;q,q)_\infty}(q;q)_\infty^{u-1}(1-q)^{-(u-1)(u-2)/2},
\\
\vartheta(u)=q^{u(u-1)/2}\Theta_q(q^u),\qquad
\Theta_q(x)=(x,q/x,q;q)_\infty,
\end{gather*}
which satisfy $\Gamma_q(1)=G_q(1)=1$ and
\begin{gather*}
\Gamma_q(u+1)=[u]\Gamma_q(u), \qquad
G_q(u+1)=\Gamma_q(u)G_q(u), \qquad \vartheta(u+1)=-\vartheta(u)=\vartheta(-u).
\end{gather*}

A partition is a finite sequence of positive integers
$\lambda=(\lambda_1,\dots,\lambda_l)$ such that
$\lambda_1\ge\cdots\ge$ $\lambda_{l}>0$.
Denote the length of the partition by $\ell(\lambda)=l$.
The conjugate partition $\lambda'=(\lambda'_1,\dots,\lambda'_{l'})$
is defined by $\lambda'_j=\sharp\{i\mid \lambda_i\ge j\}$,
$l'=\lambda_1$. We~regard a partition as a Young diagram.
Namely, we regard a partition $\lambda$ also as the subset
$\big\{(i,j)\in\Z^2\mid 1\le j\le \lambda_i,\ i\ge 1\big\}$ of~$\Z^2$,
and denote its cardinality by $|\lambda|$. We~denote the set of
all partitions by $\mathbb{Y}$.
For $\square=(i,j)\in\Z_{>0}^2$ we set $a_\lambda(\square)=\lambda_i-j$
(the arm length of~$\square$)
and $\ell_\lambda(\square)=\lambda'_j-i$ (the leg length of~$\square$).
In~the last formulas we set $\lambda_i=0$ if $i>\ell(\lambda)$
(resp.~$\lambda'_j=0$ if $j>\ell(\lambda')$). For a pair of partitions ($\lambda, \mu$)
and $u\in\C$
we set
\begin{gather*}
N_{\lambda,\mu}(u)=\prod_{\square\in \lambda}\big(
1-q^{-\ell_{\lambda}(\square)-a_\mu(\square)-1}u\big)
\prod_{\square\in \mu}\big(
1-q^{\ell_{\mu}(\square)+a_\lambda(\square)+1}u\big),
\end{gather*}
which we call a Nekrasov factor.
We~set $|\bla|=\sum_{i=1}^N\big|\lambda^{(i)}\big|$ for an element
$\bla=\big(\lambda^{(1)},\dots,\lambda^{(N)}\big)$ of~$\Y^N$.

\section{Contiguity relations}\label{S2}

\subsection[q-conformal blocks]{$\boldsymbol q$-conformal blocks}

For $\sigma\in \C^N$, set $\Delta_\sigma=\sum_{i=1}^N\big( \sigma^{(i)}\big)^2/2$.
We~define $h_i\in\C^N$, $i=1,\dots,N$, as
\begin{gather*}
h_i^{(j)}=\delta_{ij}-\frac{1}{N},
\end{gather*}
which may be regarded as the weights of the vector representation of~$\frak{sl}_N$.
We~choose the nor\-ma\-lization factor as
\begin{gather*}
\N\Bigl({\kern-1pt\theta\atop{\sigma_{2} \phantom{h_2} \sigma_{1}}}
\Bigr)=\frac{\prod_{k,k'=1}^NG_q\big(1+\sigma_2^{(k)}-\theta-\sigma_1^{(k')}\big)}
{\prod_{1\leq k<k'\leq N}G_q\big(1+\sigma_2^{(k)}-\sigma_2^{(k')}\big)
G_q\big(1-\sigma_1^{(k)}+\sigma_1^{(k')}\big)}
\end{gather*}
for $\theta\in\C$, $\sigma_i\in\C^N$, $i=1,2$.

\begin{Definition}
Let $m\in\Z_{\geq 2}$. For $\theta_k\in\C$ ($1\leq k\leq m$),
and $\sigma_k= \big(\sigma_k^{(1)}, \dots, \sigma_k^{(N)}\big) \in \C^{N}$,
$k=0, \dots , m$, such that $\sum_{j=1}^{N} \sigma_k^{(j)}=0$,
 we define the $m+2$ point $q$-conformal block by
 \begin{gather}
\mathcal{F}\left({
{\phantom{000} \theta_m\phantom{h000}
\theta_{m-1}\phantom{h0}\phantom{h0}
\cdots\phantom{h0}\phantom{h0}\theta_1\phantom{h0}}
\atop
{\sigma_{m}\phantom{h0}\sigma_{m-1}
\phantom{h0}\sigma_{m-2}\phantom{h0}
\cdots
\phantom{h_0}\sigma_{1}\phantom{h0}\sigma_{0}}}
;\,x_m,\dots,x_{1}
\right)\nonumber
\\ \qquad
=\prod_{p=1}^{m}\N\Bigl({\kern-1pt\theta_p\atop{\sigma_{p} \phantom{h_2} \sigma_{p-1}}}
\Bigr)q^{N\theta_p\Delta_{\sigma_{p}}}
\prod_{p=1}^{m}x_p^{\Delta_{\sigma_{p}}-\Delta_{N\theta_p h_1}-\Delta_{\sigma_{p-1}}}\nonumber
\\ \qquad\phantom{=}
{}\times\sum_{\bla_1,\dots, \bla_{m-1}\in\Y^N}
\prod_{p=1}^{m-1} \bigg(\frac{q^{N\theta_p} x_p}{x_{p+1}} \bigg)^{\left|\bla_p\right|}
 \frac{\prod_{p=1}^m\prod_{k,k'=1}^{N}N_{\lambda_p^{(k)}, \lambda_{p-1}^{(k')}}
\big(q^{\sigma_p^{(k)}-\theta_p-\sigma_{p-1}^{(k')}}\big)}
{\prod_{p=1}^{m-1}\prod_{k,k'=1}^{N} N_{\lambda_p^{(k)}, \lambda_{p}^{(k')}}
\big(q^{\sigma_p^{(k)}-\sigma_{p}^{(k')}}\big)}, \label{eq_qCB}
\end{gather}
where $\bla_p=\big(\lambda_p^{(1)},\dots,\lambda_p^{(N)}\big)$,
 and $\bla_0=\bla_m=(\varnothing, \dots, \varnothing)$.
\end{Definition}
Our $q$-conformal blocks can be obtained as expectation values of intertwining operators of~the~Ding--Iohara--Miki algebra~\cite{AFS}. We~omit a derivation and refer to~\cite[Section~2]{Jimbo-Nagoya-Sakai}.\footnote{If we replace $\mathcal{V}_{\theta,w}$ on p.~6 in~\cite{Jimbo-Nagoya-Sakai} by the tensor product of the $N$ Fock modules of the Ding--Iohara--Miki algebra, then we will get~\eqref{eq_qCB}, provided that the normalization is suitably chosen.} The series part of a $q$-conformal block is the 5d Nekrasov instanton partition function.

First, we relate a $4$-point conformal block to the $q$-hypergeometric series, which is defined
as follows.
\begin{Definition}
Let $\alpha_1, \dots, \alpha_N,\, \beta_1, \dots, \beta_{N-1} \in \C$.
A series\vspace{-1ex}
\begin{gather*}
 _{N} \phi_{N-1} \left( \begin{matrix}
\alpha_1 , \dots , \alpha_{N} \\
\beta_1 , \dots , \beta_{N-1} \\
\end{matrix} ;\, q, z \right) =\sum_{k=0}^\infty \frac{\big(q^{\alpha_1} ; q\big)_k \cdots \big(q^{\alpha_{N}} ; q\big)_k} {\big(q^{\beta_1}; q\big)_k \cdots \big(q^{\beta_{N-1}} ; q\big)_k (q ; q)_k}z^k
\end{gather*}
is called the $q$-hypergeometric series.
\end{Definition}

Let us recall some elementally facts about the Nekrasov factor.
\begin{Lemma}\label{lem_Young1}
For $\lambda,\mu \in \Y, w \in \C$, we have
\begin{itemize}\itemsep=0pt
\item[$(1)$] $N_{\lambda, \varnothing} \big(q^{-1}\big) \neq 0$ if and only if
 $\lambda=(1^n)$, $n=0,1,2,\dots$,
\item[$(2)$] $N_{(1^n), \varnothing}(w)=(q^{-n+1}w ; q)_n$,
\item[$(3)$] $N_{\varnothing, (1^n)}(w)=(w ; q)_n$,
\item[$(4)$] $N_{(1^n), (1^n)}(w)=\big(q^{-n}w ; q\big)_n (qw ; q)_n$.
\end{itemize}
\end{Lemma}

We~need to substitute a special value into the normalization factor as\vspace{-1ex}
\begin{gather*}
\N\left( \frac{1}{N}\atop{\sigma_{1}-h_i \phantom{h_2} \sigma_{1}}
\right).
\end{gather*}
But then it has a zero factor. We~redefine the normalization factor in this case as\vspace{-.5ex}
\begin{gather*}
\N'\left( \frac{1}{N}\atop{\sigma_2 \phantom{h_2} \sigma_2+h_i}
\right)=\lim_{\epsilon\to 0}\frac{1}{G_q(\epsilon)}\N\left( \frac{1}{N}-\epsilon\atop{\sigma_{2} \phantom{h_2} \sigma_{2}+h_i}
\right).
\end{gather*}
The $q$-conformal block with the special value above is called {\it degenerate} $q$-conformal block.
The reduction of the 5d Nekrasov partition function to the $q$-hypergeometric series was explained in~\cite{BTZ1, BTZ2}, which is stated in our notation as follows.

\begin{Proposition}\label{prop_qHGS}
 For $i=1,\dots, N$, we have
\begin{gather*}
\mathcal{F}\left({
\frac{1}{N}\phantom{0000}\theta_1
\atop
\sigma_2\phantom{00}\sigma_2+h_i\phantom{00}\sigma_0
};\,x_1,x_2
\right)
\\ \qquad
{}=\mathcal{N}q^{\Delta_{\sigma_2}}
\bigg(\frac{q^{N\theta_1}x_2}{x_1}\bigg)^{\sigma_2^{(i)}+(1-1/N)/2}
\frac{\prod_{k=1}^N\Gamma_q\big(1-\frac{1}{N}+\sigma_2^{(i)}-\theta_1-\sigma_0^{(k)}\big)}
{\prod_{k=1}^N\Gamma_q\big(1+\sigma_2^{(i)}-\sigma_2^{(k)}\big)}
\\ \qquad\phantom{=}
{}\times
{}_{N}\phi_{N-1}\left(\begin{matrix}
\big\{1-\tfrac1N+\sigma_{2}^{(i)}-\theta_1-\sigma_{0}^{(k)}\big\}_{k=1}^N
 \\[.5ex]
\big\{1+\sigma_{2}^{(i)}-\sigma_{2}^{(k)}\big\}_{k=1,\,k\neq i}^N
\end{matrix} ;\, q, \frac{q^{N\theta_1}x_2}{x_1} \right)\!,
\\
\mathcal{F}\left({
\theta_1\phantom{0000}\frac{1}{N}
\atop
\sigma_2\phantom{00}\sigma_0-h_i\phantom{00}\sigma_0
};\,x_2,x_1
\right)
\\ \qquad
{}=\mathcal{N}q^{\Delta_{\sigma_0}}\bigg(\frac{qx_1}{x_2}\bigg)^{-\sigma_0^{(i)}+(1-1/N)/2}
\frac{\prod_{k=1}^N\Gamma_q\big(1-\frac{1}{N}+\sigma_2^{(k)}-\theta_1-\sigma_0^{(i)}\big)}
{\prod_{k=1}^N\Gamma_q\big(1+\sigma_0^{(k)}-\sigma_0^{(i)}\big)}
\\ \qquad\phantom{=}
{}\times
{}_{N}\phi_{N-1}\left(\begin{matrix}
\big\{1-\frac1N+\sigma_{2}^{(k)}-\theta_2-\sigma_{0}^{(i)}\big\}_{k=1}^N
\\[.5ex]
\big\{1+\sigma_{0}^{(k)}-\sigma_{0}^{(i)}\big\}_{k=1,\,k\neq i}^N
\end{matrix} ;\, q, \frac{qx_1}{x_2} \right)\!,
\end{gather*}
where
\begin{gather*}
\mathcal{N}=q^{N\theta_1\Delta_{\sigma_2}}x_1^{-(1-1/N)/2}x_2^{\Delta_{\sigma_2}-\Delta_{N\theta_1h_1}-\Delta_{\sigma_0}}
\\ \hphantom{\mathcal{N}=}
{}\times \frac{\prod_{k,k'=1}^NG_q\big(1-\frac{1}{N}+\sigma_2^{(k)}-\theta_1-\sigma_0^{(k')}\big)}
{\prod_{1\leq k<k'\leq N}G_q\big(1+\sigma_2^{(k)}-\sigma_2^{(k')}\big)
G_q\big(1-\sigma_0^{(k)}+\sigma_0^{(k')}\big)}\!.
\end{gather*}
\end{Proposition}
\begin{proof}
By the definition of~$q$-conformal blocks and Lemma~\ref{lem_Young1}, we obtain
the desired re\-sults.
\end{proof}

\subsection{Connection problem}

The connection formula of the $q$-hypergeometric series is well known. See the equation (4.5.2) on p.~121 in~\cite{Gasper-Rahman}
for example. From Proposition~\ref{prop_qHGS}, the connection formula of
the degenerate $4$-point $q$-conformal block reads as
\begin{gather}
\mathcal{F}\left({
\frac{1}{N}\phantom{0000}\theta_1
\atop
\sigma_2\phantom{00}\sigma_2+h_i\phantom{00}\sigma_0};\,x_1,x_2
\right)\nonumber
\\ \qquad
{}=\sum_{j=1}^N\mathcal{F}\left({
\theta_1\phantom{0000}\frac{1}{N}
\atop
\sigma_2\phantom{00}\sigma_0-h_j\phantom{00}\sigma_0
};\,x_2,x_1
\right) \B_{j,i}\left[\begin{matrix}
 \theta_1&\frac{1}{N}
 \\
 \sigma_2&\sigma_0\end{matrix}
\Bigl| \frac{x_2}{x_1}\right]q^{N\theta_1^2/2-\theta_1/2}\bigg(\frac{x_2}{x_1}\bigg)^{\theta_1},
\label{eq_connection_problem_deg_4pt}
\end{gather}
where $\B\left[\begin{smallmatrix}
 \theta_1&\frac{1}{N}
 \\
 \sigma_2&\sigma_0\end{smallmatrix}
\Bigl| x\right]$ is the $N$ by $N$ matrix given by
\begin{gather*}
\B_{j,i}\left[\begin{matrix}
 \theta_1&\frac{1}{N}
 \\
 \sigma_2&\sigma_0\end{matrix}
\Bigl| x\right]
=\frac{\vartheta\big(1\!-\frac1N+\sigma_2^{(i)}\!+(N\!-1)\theta_1\!-\sigma_0^{(j)}+u\big)}
{\vartheta(N\theta_1+u)}
\frac{\prod_{k=1,\,k\neq i}^N\vartheta\big(\frac1N\!-\sigma_2^{(k)}\!+\theta_1+\sigma_0^{(j)}\big)}{\prod_{k=1,k\neq j}^N\vartheta\big(\sigma_0^{(j)}\!-\sigma_0^{(k)}\big)}
\end{gather*}
with $x=q^u$. We~mention that the connection problem of the $q$-hypergeometric series was used to derive the three-point function for the $q$-deformed correlators where conformal blocks are controlled by the $q$-Virasoro algebra~\cite{NPP}.

The connection matrix $\B\left[\begin{smallmatrix}
 \theta_1&\frac{1}{N}
 \\
 \sigma_2&\sigma_0\end{smallmatrix}
\Bigl| x\right]
$ enjoys the following properties.
\begin{Lemma}
We~have
\begin{gather}
\B\left[\begin{matrix}
 \theta_1&\frac{1}{N} \\ \sigma_2&\sigma_0\end{matrix}
\Bigl| qx\right]
=\B\left[\begin{matrix}
 \theta_1&\frac{1}{N} \\ \sigma_2&\sigma_0\end{matrix}
\Bigl| x\right],
\label{eq_q-shift_on_x_connection_matrix}
\\[.5ex]
\B\left[\begin{matrix}
 \theta_1+1&\frac{1}{N} \\ \sigma_2&\sigma_0\end{matrix}
\Bigl| x\right]
=(-1)^N\B\left[\begin{matrix}
 \theta_1&\frac{1}{N} \\ \sigma_2&\sigma_0\end{matrix}
\Bigl| x\right],
\label{eq_shifts_connection_matrix_theta+1}
\\[.5ex]
\B\left[\begin{matrix}
 \theta_1&\frac{1}{N} \\ \sigma_2+h_i&\sigma_0+h_j\end{matrix}
\Bigl| x\right]
=\B\left[\begin{matrix}
 \theta_1&\frac{1}{N} \\ \sigma_2&\sigma_0\end{matrix}
\Bigl| x\right],\qquad i,j=1,\dots,N,
\label{eq_shifts_conncetion_matrix_sigma}
\\[.5ex]
\B\left[\begin{matrix}
 \theta_1+1/N&\frac{1}{N} \\ \sigma_2&\sigma_0+h_i\end{matrix}
\Bigl| x\right]
=-\B\left[\begin{matrix}
 \theta_1&\frac{1}{N} \\ \sigma_2&\sigma_0\end{matrix}
\Bigl| x\right],\qquad i=1,\dots, N,
\label{eq_shifts_conncetion_matrix_theta}
\\[.5ex]
\det \B\left[\begin{matrix}
 \theta_1&\frac{1}{N} \\ \sigma_2&\sigma_0\end{matrix}
\Bigl| x\right]=(-1)^{N-1}\frac{\vartheta(u)}{\vartheta(N\theta_1+u)}
\prod_{1\leq i<j\leq N}\frac{\vartheta\big(\sigma_2^{(i)}-\sigma_2^{(j)}\big)}
{\vartheta\big(\sigma_0^{(i)}-\sigma_0^{(j)}\big)}.\label{eq_det_connction_matrix}
\end{gather}
\end{Lemma}
\begin{proof}
The identities~\eqref{eq_q-shift_on_x_connection_matrix},~\eqref{eq_shifts_connection_matrix_theta+1},
\eqref{eq_shifts_conncetion_matrix_sigma}, and~\eqref{eq_shifts_conncetion_matrix_theta} are
immediate consequences of the definition of the connection matrix. The expression of the
determinant~\eqref{eq_det_connction_matrix} of the connection matrix can be verified by straightforward
calculations as follows.

By definition, we have
\begin{gather*}
\det\B\left[\begin{matrix}
\theta_1&\frac{1}{N}	\\	\sigma_2&\sigma_0\end{matrix}
\Bigl| x\right]=\frac{1}{\vartheta(N\theta_1+u)^N\prod_{i=1}^N\prod_{k\neq i}\vartheta\big(\sigma_0^{(i)}-\sigma_0^{(k)}\big)}	\det(a_{ij}),
\end{gather*}
where
\begin{gather*}
a_{ij}=\vartheta\big(1-\tfrac1N+\sigma_2^{(j)}+(N-1)\theta_1-\sigma_0^{(i)}+u\big)\prod_{k\neq j}^N\vartheta\big(\tfrac1N-\sigma_2^{(k)}+\theta_1+\sigma_0^{(i)}\big).
\end{gather*}
Put $x_{ij}=1-\tfrac1N+\sigma_2^{(j)}+(N-1)\theta_1-\sigma_0^{(i)}+u$ and
$y_{ij}=\tfrac1N-\sigma_2^{(j)}+\theta_1+\sigma_0^{(i)}$.
Using the relation on the theta function:
\begin{gather*}
\vartheta(x+y)\vartheta(x-y)\vartheta(u+v)\vartheta(u-v)
-\vartheta(x+v)\vartheta(x-v)\vartheta(u+y)\vartheta(u-y)
\\ \qquad
{}=\vartheta(x+u)\vartheta(x-u)\vartheta(y+v)\vartheta(y-v),
\end{gather*}
we obtain for $i,j\geq 2$
\begin{gather*}
a_{11}a_{ij}-a_{i1}a_{1j}=\left(\vartheta(x_{11})\vartheta(y_{1j})\vartheta(x_{ij})\vartheta(y_{i1})
-\vartheta(x_{i1})\vartheta(y_{ij})\vartheta(x_{1j})\vartheta(y_{11})\right)
\prod_{k\neq 1,j}\vartheta(y_{1k})	\vartheta(y_{ik})
\\ \hphantom{a_{11}a_{ij}-a_{i1}a_{1j}}
{}=\vartheta(1+N\theta_1+u)\vartheta\big(1-\tfrac2N+(N-2)\theta_1+u+\sigma_2^{(1)}+\sigma_2^{(j)}
-\sigma_0^{(1)}-\sigma_0^{(i)}\big)
\\ \hphantom{a_{11}a_{ij}-a_{i1}a_{1j}=}
{}\times \vartheta\big(\sigma_2^{(1)}-\sigma_2^{(j)}\big)\vartheta\big({-}\sigma_0^{(1)}+\sigma_0^{(i)}\big)
\prod_{k\neq 1,j}\vartheta(y_{1k})\vartheta(y_{ik}).
\end{gather*}
Hence, the determinant becomes
\begin{gather*}
\det\B\left[\begin{matrix}
\theta_1&\frac{1}{N}\\	\sigma_2&\sigma_0\end{matrix}\Bigl| x\right]
=\frac{\prod_{k=2}^N\vartheta\big(\sigma_2^{(1)}-\sigma_2^{(k)}\big)}
{\vartheta(N\theta_1+u)\vartheta(x_{11})^{N-2}\prod_{i=2}^N\prod_{k\neq i}\vartheta\big(\sigma_0^{(i)}-\sigma_0^{(k)}\big)}	\det(b_{ij}),
\end{gather*}
where
\begin{gather*}
b_{ij}=\vartheta\big(1-\tfrac2N+(N-2)\theta_1+u+\sigma_2^{(1)}+\sigma_2^{(j)}-\sigma_0^{(1)}-\sigma_0^{(i)}\big)
\!\prod_{k\neq 1,j}\!\vartheta(y_{ik}),\qquad 2\leq i,j\leq N.
\end{gather*}
In~the same way, by computing $b_{22}b_{ij}-b_{i2}b_{2j}$, the determinant $\det(b_{ij})$ reduces to a determinant of size $N-2$ whose entries are products of the theta functions. Repeating the computation, we~obtain the desired result~\eqref{eq_det_connction_matrix}.
\end{proof}

We note that the connection matrix has periodicity on the variable $x$ and
parameters $\sigma_0$, $\sigma_2$, and $\theta_1$. Furthermore, the determinant
has a simple form in terms of the theta function.

We want to prove the following connection formula for $6$-point degenerate $q$-conformal block
\begin{gather}
\mathcal{F}\left(
{\theta_4\phantom{000} \frac{1}{N}\phantom{00000} \theta_2 \phantom{000} \theta_1
\phantom{}
\atop
\sigma_4 \phantom{00} \sigma_3\phantom{00} \sigma_3+h_i \phantom{00}
\sigma_1\phantom{00} \sigma_0}
 ;\, x_4,x_2,x_3,x_1\right) \nonumber
\\ \qquad
{}=\sum_{j=1}^{N} \mathcal{F}
\left(
{\theta_4\phantom{000} \theta_2\phantom{00000} \frac{1}{N} \phantom{000} \theta_1
\phantom{}
\atop
\sigma_4 \phantom{00} \sigma_3\phantom{00} \sigma_1-h_j \phantom{00}
\sigma_1\phantom{00} \sigma_0}
 ;\, x_4,x_3,x_2,x_1
 \right)\B_{j,i}\left[\begin{matrix}
 \theta_2&\frac{1}{N}\nonumber
 \\
 \sigma_3&\sigma_1\end{matrix}
\Bigl| \frac{x_3}{x_2}\right]
\\ \qquad\phantom{=}
{}\times q^{N\theta_2^2/2-\theta_2/2}\bigg(\frac{x_3}{x_2}\bigg)^{\theta_2}.\label{eq_CP_6qCB}
\end{gather}
The constant with respect to the variables $x_1$, $x_4$
in the $6$-point degenerate $q$-conformal block is $q$-hypergeometric series.
In~this case, the connection problem is solved. We~solve the connection problem
\eqref{eq_CP_6qCB} by solving connection problem of each coefficient of~$x_1$, $x_4$ in the $6$-point degenerate $q$-conformal block. We~consider the following series obtained by taking the coefficient of~$x_1^{|\bla|}x_4^{-|\bnu|}$ in the $6$-point $q$-conformal block:
\begin{gather}\label{coefficient_4pt_qCB}
S_{\bla, \bnu}\left({
\theta_3\phantom{000}\theta_2
\atop
\sigma_3\phantom{00}\sigma_2\phantom{00}\sigma_1
};\,x_3,x_2
\right)
=\sum_{\bmu \in \Y^{N}} \bigg(\frac{q^{N\theta_2} x_2}{x_3}\bigg)^{|\bmu|} \prod_{k,k'=1}^{N}\frac{N_{\lambda_k, \mu_{k'}}(w_{k, k'}) N_{{\mu_k}, \nu_{k'}}(z_{k, k'})}{ N_{\mu_k, \mu_{k'}}(u_{k, k'})}.
\end{gather}
Here we put
\begin{gather*}
w_{k, k'}=q^{\sigma_k^{(3)}-\theta_3-\sigma_{k'}^{(2)}}, \qquad
z_{k, k'}=q^{\sigma_k^{(2)}-\theta_2-\sigma_{k'}^{(1)}}, \qquad
u_{k, k'}=q^{\sigma_k^{(2)}-\sigma_{k'}^{(2)}}.
\end{gather*}
We note that for $\bla=\bnu=(\varnothing,\dots, \varnothing)$,
\begin{gather*}
\mathcal{F}\left({
\theta_3\phantom{000}\theta_2
\atop
\sigma_3\phantom{00}\sigma_2\phantom{00}\sigma_1
};\,x_3,x_2
\right)
\\ \qquad
{}=\prod_{p=2}^{3}\N\Bigl({\kern-1pt\theta_p\atop{\sigma_{p} \phantom{h_2} \sigma_{p-1}}}
\Bigr)q^{N\theta_p\Delta_{\sigma_{p}}}
x_p^{\Delta_{\sigma_{p}}-\Delta_{N\theta_p h_1}-\Delta_{\sigma_{p-1}}}\,
S_{\varnothing, \varnothing}\left({
\theta_3\phantom{000}\theta_2
\atop
\sigma_3\phantom{00}\sigma_2\phantom{00}\sigma_1
};\,x_3,x_2
\right)\!.
\end{gather*}

From Proposition~\ref{prop_qHGS} if we set $\theta_3=1/N$,
$\sigma_2=\sigma_3+h_i$ or $\theta_2=1/N$,
$\sigma_2=\sigma_1-h_i$ in~\eqref{coefficient_4pt_qCB}, then
$S_{\varnothing,\varnothing}$ becomes the $q$-hypergeomeric series ${}_N\phi_{N-1}$.
In~that case, $\bmu$ in the right hand side of~\eqref{coefficient_4pt_qCB}
restricts to $\mu_i=(1^n)$ ($n=0,1,2,\dots$), $\mu_k=\varnothing$ ($k\neq i$).
For the general $\bla$, $\bnu$, a tuple $\bmu$ of~partitions
also takes the following particular forms.

Let us introduce the following operations on partitions $\lambda\in \Y$:
\begin{gather}
\bar{\lambda}=(\lambda_1-1,\dots, \lambda_{\ell(\lambda)}-1)\,,
\label{lambda-bar}
\\
r_n(\lambda)=(\lambda_1+1,\dots, \lambda_n+1,\lambda_{n+2},\dots),\qquad
n\in\Z_{\ge0}.
\label{rn}
\end{gather}
Here we identify a partition $\lambda$ with a sequence
$\big(\lambda_1,\dots,\lambda_{\ell(\lambda)},0,0,\dots\big)$.

Let us recall the following lemma.
\begin{Lemma}
\label{lem_Young2}
For $\lambda,\mu \in \Y$, we have
\begin{itemize}\itemsep=0pt

\item[$(1)$] $N_{\lambda, \mu} (1) \neq 0$ if and only if $\lambda=\mu$,
\item[$(2)$] $N_{\lambda, \mu} (q^{-1}) \neq 0 $ if and only if
$\lambda=r_n(\mu)$, $n=0,1,2,\dots$,
\item[$(3)$] $N_{\lambda, \mu} (u)=N_{\mu',\lambda'} (u)$.
\end{itemize}
\end{Lemma}
{\sloppy\noindent
Lemma~\ref{lem_Young2}(1) is easily obtained from the definition of the Nekrasov factor and Lem\-ma~\ref{lem_Young2}(2),~(3) were proved in~\cite{Jimbo-Nagoya-Sakai} as Lemmas A.2 and A.3, respectively.

}

Suppose that $\theta_3=1/N$, $\sigma_2=\sigma_3+h_i$
in~\eqref{coefficient_4pt_qCB}, then the following holds:
\begin{gather*}
w_{i, i}=q^{\sigma_3^{(i)}-1/N-\sigma_2^{(i)}}=q^{-1}, \qquad
w_{k, k}=q^{\sigma_3^{(k)}-1/N-\sigma_2^{(k)}}=1, \qquad k \neq i.
\end{gather*}
Hence by Lemma~\ref{lem_Young2}(1), (2), we obtain
\begin{gather*}
\mu_i=r_n(\nu_i),\qquad \mu_k=\nu_k, \qquad k\neq i.
\end{gather*}
Similarly, suppose that $\theta_2=1/N$, $\sigma_2=\sigma_1-h_i$
in~\eqref{coefficient_4pt_qCB}, then the following holds:
\begin{gather*}
z_{i, i}=q^{\sigma_2^{(i)}-1/N-\sigma_1^{(i)}}=q^{-1}, \qquad
z_{k, k}=q^{\sigma_2^{(k)}-1/N-\sigma_1^{(k)}}=1, \qquad k \neq i.
\end{gather*}
Hence by Lemma~\ref{lem_Young2}(1), (2), we obtain
\begin{gather*}
\mu_i=r_n(\lambda_i),\qquad \mu_k=\lambda_k, \qquad k\neq i.
\end{gather*}

We also recall the following lemma.
\begin{Lemma}[{\cite[Lemma A.4]{Jimbo-Nagoya-Sakai}}]\label{lem-N-factors}
Let $\lambda,\mu \in \Y$ and $n\in\Z_{\ge0}$.
Using the notation~\eqref{lambda-bar} and~\eqref{rn}, we set
\begin{gather*}
\ell=\ell(\lambda),\qquad
k=\ell(\mu),\qquad
\eta=r_n(\lambda), \qquad
\gamma=(r_n(\mu'))',
\\
\tilde{\eta}=
\begin{cases}
\bar{\eta}, & n\le \ell-1,\\
\big(\lambda,1^{n-\ell+1}\big), & n\ge \ell.
\end{cases}
\end{gather*}
Then we have
\begin{gather*}
\frac{N_{\eta,\lambda}(q^{-1})}{N_{\eta,\eta}(1)}=
\frac{N_{\tilde{\eta},\bar{\lambda}}(q^{-1})}
{N_{\tilde{\eta},\tilde{\eta}}(1)}
\big(1-q^{|\tilde{\eta}|-|\lambda|}\big),
\\
N_{\mu,\eta}(u)=
N_{\mu,\tilde{\eta}}(q^{-1}u)
\prod_{j=1}^{\mu_\ell}
\frac{1-q^{j-1}u}{1-q^{-\ell_\mu(\ell,j)+j-2}u}
\prod_{i=1}^{\ell-1}\big(1-q^{\ell-i+a_\mu(i,1)}u\big),
\\
N_{\mu,\lambda}(u)
=N_{\mu,\bar{\lambda}}(q^{-1}u)
\prod_{j=1}^{\mu_{\ell+1}}
\frac{1-q^{j-1}u}{1-q^{-\ell_\mu(\ell+1,j)+j-2}u}
\prod_{i=1}^\ell\big(1-q^{\ell-i+a_\mu(i,1)+1}u\big),
\\
\frac{N_{\mu,\lambda}(u)}{N_{\mu,\eta}(qu)}
=\frac{N_{\mu,\bar{\lambda}}(q^{-1}u)}{N_{\mu,\tilde{\eta}}(u)}(1-u),
\\
\frac{N_{\mu,\lambda}(u)}{N_{\mu,\eta}(qu)}
=\frac{N_{\bar{\mu},\lambda}(qu)}{N_{\bar{\mu},\eta}(q^2u)}
\frac{1-q^{|\eta|-|\lambda|+1-k}u}{1-qu},
\\
N_{\gamma,\lambda}(u)=
N_{\gamma,\bar{\lambda}}(q^{-1}u)
\big(1-q^{\ell+|\gamma|-|\mu|-1}u\big)
\prod_{j=1}^{\mu_\ell}
\frac{1-q^{j-1}u}{1-q^{-\ell_\mu(\ell,j)+j-2}u}
\prod_{i=1}^{\ell-1}\big(1-q^{\ell-i+a_\mu(i,1)}u\big).
\end{gather*}
\end{Lemma}

We denote by $T_{q,z}$, the $q$-shift operator with respect to the variable $z$, namely, $T_{q,z}(f(z))=f(qz)$
for a function $f(z)$. For a partition $\lambda=(\lambda_1,\lambda_2,\dots)$, we define
\begin{gather*}
	\lambda+1^m=(\lambda_1+1,\lambda_2+1,\dots,\lambda_m+1,\lambda_{m+1},\dots).
\end{gather*}

\begin{Theorem}\label{thm_contiguity_relations_S}
For $i,j$ $(1\leq i,j\leq N)$, let $m\geq \ell(\lambda_j)$, and
\begin{gather*}
\hat\bla=\big(\lambda_1,\dots,\lambda_j+1^m,\dots,\lambda_N\big),
\\
A_{\bla,\bnu}^{(j)}(\theta_2,\sigma_1,\sigma_3)=
\prod_{k=1}^N\frac{N_{\nu_k,\hat\lambda_j}
\big(q^{\sigma_3^{(k)}-\theta_2-\sigma_1^{(j)}+1-1/N}\big)}
{N_{\nu_k,\lambda_j}\big(q^{\sigma_3^{(k)}-\theta_2-\sigma_1^{(j)}-1/N}\big)},
\end{gather*}
we have
\begin{gather}
S_{\hat\bla, \bnu}\left({
\frac{1}{N}\phantom{000}\theta_2-\frac{1}{N}
\atop
\sigma_3\phantom{00}\sigma_3+h_i\phantom{00}\sigma_1-h_j
};\,x_2,x_3
\right)\nonumber
\\ \qquad
{}=q^{m-|\bnu| +1-1/N+\sigma_3^{(i)}-\theta_2-\sigma_1^{(j)}}A_{\bla,\bnu}^{(j)}(\theta_2,\sigma_1,\sigma_3)
\frac{1}
{1-q^{\sigma_3^{(i)}-\theta_2-\sigma_1^{(j)}+1-1/N}}\nonumber
\\ \qquad\phantom{=}
{}\times
\big(q^{|\bnu|-m-1+1/N-\sigma_3^{(i)}+\theta_2+\sigma_1^{(j)}}T_{q,x_2}-1\big)S_{\bla, \bnu}\left({
\frac{1}{N}\phantom{000}\theta_2
\atop
\sigma_3\phantom{00}\sigma_3+h_i\phantom{00}\sigma_1
};\,x_2,x_3
\right)\!,\label{eq_S1}
\\
S_{\hat\bla, \bnu}\left({
\theta_2-\frac{1}{N}\phantom{000}\frac{1}{N}
\atop
\sigma_3\phantom{00}\sigma_1-h_i-h_j\phantom{00}\sigma_1-h_j
};\,x_3,x_2
\right)=\left(\frac{qx_2}{x_3}\right)^m
A_{\bla,\bnu}^{(j)}(\theta_2,\sigma_1,\sigma_3)\nonumber
\\ \qquad\phantom{=}
{}\times
\frac{1-q^{-m-|\bla|+\sigma_1^{(j)}-\sigma_1^{(i)}}T_{q,x_2}}{1-q^{\sigma_1^{(j)}-\sigma_1^{(i)}}}S_{\bla, \bnu}\left({
\theta_2\phantom{000}\frac{1}{N}
\atop
\sigma_3\phantom{00}\sigma_1-h_i\phantom{00}\sigma_1
};\,x_3,x_2
\right),\qquad i\neq j, \label{eq_S2}
\end{gather}
\begin{gather}
S_{\hat\bla, \bnu}\left({
\theta_2-\frac{1}{N}\phantom{000}\frac{1}{N}
\atop
\sigma_3\phantom{00}\sigma_1-2h_i\phantom{00}\sigma_1-h_i
};\,x_3,x_2
\right)\nonumber
\\ \qquad
{}=\bigg(\frac{qx_2}{x_3}\bigg)^{m-1}
A_{\bla,\bnu}^{(j)}(\theta_2,\sigma_1,\sigma_3)\prod_{k=1}^N
\frac{1-q^{\sigma_1^{(k)}-\sigma_1^{(i)}+1}}
{1-q^{\sigma_3^{(k)}-\theta_2-\sigma_1^{(j)}+1-1/N}}
 \frac{1-q^{-m-|\bla|}T_{q,x_2}}{1-q}\nonumber
 \\ \qquad\phantom{=}
{} \times S_{\bla, \bnu}\left({
\theta_2\phantom{000}\frac{1}{N}
\atop
\sigma_3\phantom{00}\sigma_1-h_i\phantom{00}\sigma_1
};\,x_3,x_2
\right)\!.\label{eq_S3}
\end{gather}
\end{Theorem}
\begin{proof}
We can verify the equations~\eqref{eq_S1},~\eqref{eq_S2},
and~\eqref{eq_S3} by direct computations using Lemma
\ref{lem-N-factors}. We~omit the calculation.
\end{proof}

We note that if $\bla=\bnu=(\varnothing,\dots,\varnothing)$,
then we can take $m=0$ in Theorem~\ref{thm_contiguity_relations_S} and then
$\hat \bla=(\varnothing,\dots,\varnothing)$. In~that case,
the equations~\eqref{eq_S1},~\eqref{eq_S2},
and~\eqref{eq_S3} are equal to~the contiguity relations for
the $q$-hypergeometric series. Hence, we may regard
the contiguity relations~\eqref{eq_S1},~\eqref{eq_S2},
and~\eqref{eq_S3} as a generalization of the
contiguity relation of the $q$-hy\-per\-geometric series.

Let us denote the coefficient of~$x_1^{|\bla|}x_4^{-|\bnu|}$ in the $6$-point $q$-conformal block by
\begin{gather*}
\F_{\bla, \bnu}\left({
\theta_3\phantom{000}\theta_2
\atop
\sigma_3\phantom{00}\sigma_2\phantom{00}\sigma_1
};\,x_3,x_2
\right)
\\ \qquad
{}=\prod_{p=2}^3\N\Bigl({\kern-1pt\theta_p\atop{\sigma_{p} \phantom{h_2} \sigma_{p-1}}}
\Bigr)q^{N\theta_p\Delta_{\sigma_{p}}}
x_p^{\Delta_{\sigma_{p}}-\Delta_{N\theta_p h_1}-\Delta_{\sigma_{p-1}}}
 x_3^{|\bnu|}x_2^{-|\bla|}
S_{\bla, \bnu}\left({
\theta_3\phantom{000}\theta_2
\atop
\sigma_3\phantom{00}\sigma_2\phantom{00}\sigma_1
};\,x_3,x_2
\right)\!.
\end{gather*}
We note that if $\bla=\bnu=(\varnothing,\dots,\varnothing)$, then it is a 4-point $q$-conformal block, namely
\begin{gather*}
\F_{\varnothing, \varnothing}\left({
\theta_3\phantom{000}\theta_2
\atop
\sigma_3\phantom{00}\sigma_2\phantom{00}\sigma_1
};\,x_3,x_2
\right)
=\F\left({
\theta_3\phantom{000}\theta_2
\atop
\sigma_3\phantom{00}\sigma_2\phantom{00}\sigma_1
};\,x_3,x_2
\right)\!.
\end{gather*}

\begin{Corollary}\label{cor_contiguity_relation_F}
We have
\begin{gather*}
\F_{\hat\bla, \bnu}\left({
\frac{1}{N}\phantom{000}\theta_2-\frac{1}{N}
\atop
\sigma_3\phantom{00}\sigma_3+h_i\phantom{00}\sigma_1-h_j
};\,x_2,x_3
\right)
\\ \qquad
{}=C
\big(q^{-m+\theta_2+\sigma_1^{(j)}}T_{q,x_2}-1\big)\F_{\bla, \bnu}\left({
\frac{1}{N}\phantom{000}\theta_2
\atop
\sigma_3\phantom{00}\sigma_3+h_i\phantom{00}\sigma_1
};\,x_2,x_3
\right)\!,
\\
 \F_{\hat\bla, \bnu}\left({
\theta_2-\frac{1}{N}\phantom{000}\frac{1}{N}
\atop
\sigma_3\phantom{00}\sigma_1-h_i-h_j\phantom{00}\sigma_1-h_j
};\,x_3,x_2
\right)
\\ \qquad
{}=C q^{\theta_2}\left(\frac{x_3}{qx_2}\right)^{1/N}
\big(1-q^{-m+\sigma_1^{(j)}}T_{q,x_2}\big)\F_{\bla, \bnu}\left({
\theta_2\phantom{000}\frac{1}{N}
\atop
\sigma_3\phantom{00}\sigma_1-h_i\phantom{00}\sigma_1
};\,x_3,x_2
\right)\!,
\end{gather*}
where
\begin{align*}
C= q^{m-|\bnu|-\Delta_{\sigma_3} +(1-1/N)/2-\theta_2-\sigma_1^{(j)}}
x_3^{-m+(N\theta_2-1)(1-1/N)+\sigma_1^{(j)}}
A_{\bla,\bnu}^{(j)}(\theta_2,\sigma_1,\sigma_3).
\end{align*}
\end{Corollary}

Corollary~\ref{cor_contiguity_relation_F} enable us to reduce the connection problem
of degenerate $\F_{\hat \bla, \bnu}(x_2)$ to the connection problem of degenerate $\F_{\bla, \bnu}(x_2)$. Firstly, we let $C
\big(q^{-m+\theta_2+\sigma_1^{(1)}}T_{q,x_2}-1\big)$ act on the connection formula~\eqref{eq_connection_problem_deg_4pt} of the degenerate 4-point $q$-conformal block. Then we have the connection formula
\begin{gather*}
\F_{\hat\bla, \varnothing}\left({
\frac{1}{N}\phantom{000}\theta_2-\frac{1}{N}
\atop
\sigma_3\phantom{00}\sigma_3+h_i\phantom{00}\sigma_1-h_1
};\,x_2,x_3
\right)
=-\sum_{j=1}^N\mathcal{F}_{\hat\bla,\varnothing}\left({
\theta_2-\frac{1}{N}\phantom{0000}\frac{1}{N}
\atop
\sigma_3\phantom{00}\sigma_1-h_j-h_1\phantom{00}\sigma_1-h_1
};\,x_2,x_1
\right)
\\ \hphantom{\F_{\hat\bla, \varnothing}\left({
\frac{1}{N}\phantom{000}\theta_2-\frac{1}{N}
\atop
\sigma_3\phantom{00}\sigma_3+h_i\phantom{00}\sigma_1-h_1
};\,x_2,x_3
\right)
=}
{}\times\B_{j,i}\left[\begin{matrix}
 \theta_2&\frac{1}{N}
 \\
 \sigma_3&\sigma_1\end{matrix}
\Bigl| \frac{x_3}{x_2}\right]
q^{N\theta_2^2/2-3\theta_2/2+1/N}\bigg(\frac{x_3}{x_2}\bigg)^{\theta_2-1/N}\!\!,
\end{gather*}
where $\hat \bla=(1^m,\varnothing,\dots,\varnothing)$. Substituting $\theta_2+1/N$, $\sigma_1-h_1$ into $\theta_2$, $\sigma_1$, respectively, we obtain
\begin{gather*}
\F_{\hat\bla, \varnothing}\left({
\frac{1}{N}\phantom{000}\theta_2
\atop
\sigma_3\phantom{00}\sigma_3+h_i\phantom{00}\sigma_1
};\,x_2,x_3
\right)
=\sum_{j=1}^N\mathcal{F}_{\hat\bla,\varnothing}\left({
\theta_2\phantom{0000}\frac{1}{N}
\atop
\sigma_3\phantom{00}\sigma_1-h_j\phantom{00}\sigma_1
};\,x_2,x_1
\right)
\\ \hphantom{\F_{\hat\bla, \varnothing}\left({
\frac{1}{N}\phantom{000}\theta_2
\atop
\sigma_3\phantom{00}\sigma_3+h_i\phantom{00}\sigma_1
};\,x_2,x_3
\right)
=}
{}\times\B_{j,i}\left[\begin{matrix}
 \theta_2&\frac{1}{N}
 \\
 \sigma_3&\sigma_1\end{matrix}
\Bigl| \frac{x_3}{x_2}\right]q^{N\theta_2^2/2-\theta_2/2}\left(\frac{x_3}{x_2}\right)^{\theta_2}
\end{gather*}
due to the periodicity~\eqref{eq_shifts_conncetion_matrix_theta}
 of the connection matrix $\B\left[\begin{smallmatrix}
 \theta_2&\frac{1}{N}
 \\
 \sigma_3&\sigma_1\end{smallmatrix}
\bigl| x\right]
$. In~this way, by using Corollary~\ref{cor_contiguity_relation_F} repeatedly and by the symmetry of~$\mathcal{F}_{\bla,\bnu}$ with respect to $\bla,\bnu$, we obtain the connection formulas for any degenerate $\mathcal{F}_{\bla,\bnu}$ whose connection matrix
 is the same of degenerate $\mathcal{F}_{\varnothing,\varnothing}$.
 Therefore, we obtain
\begin{Theorem}\label{thm_CF_qCB}
We have
\begin{gather*}
\mathcal{F}\left(
{\theta_4\phantom{000} \frac{1}{N}\phantom{00000} \theta_2 \phantom{000} \theta_1
\phantom{}
\atop
\sigma_4 \phantom{00} \sigma_3\phantom{00} \sigma_3+h_i \phantom{00}
\sigma_1\phantom{00} \sigma_0}
 ;\,
 x_4,x_2,x_3,x_1\right)
\\ \qquad
{}=\sum_{j=1}^{N} \mathcal{F}
\left(
{\theta_4\phantom{000} \theta_2\phantom{00000} \frac{1}{N} \phantom{000} \theta_1
\phantom{}
\atop
\sigma_4 \phantom{00} \sigma_3\phantom{00} \sigma_1-h_j \phantom{00}
\sigma_1\phantom{00} \sigma_0}
 ;\,
 x_4,x_3,x_2,x_1
 \right)\B_{j,i}\left[\begin{matrix}
 \theta_2&\frac{1}{N}
 \\
 \sigma_3&\sigma_1\end{matrix}
\Bigl| \frac{x_3}{x_2}\right]
\\ \qquad\phantom{=}
{}\times q^{N\theta_2^2/2-\theta_2/2}\bigg(\frac{x_3}{x_2}\bigg)^{\theta_2}.
\end{gather*}
\end{Theorem}

\section{CFT construction}\label{S3}

In~this section, we construct a monodromy invariant fundamental solution of a linear $q$-difference system as a Fourier transform of~$q$-Nekrasov functions multiplied by some factors. Our fundamental system degenerates by the limit $q\to 1$ to a monodromy invariant fundamental system of a Fuchsian system of rank $N$ with spectral type
\begin{gather*}
\left(1^N\right),\quad \left(1^N\right), \quad \left(N-1,1\right),\quad \dots,\quad \left(N-1,1\right).
\end{gather*}

Another monodromy invariant fundamental solution of this Fuchsian system was constructed
in~\cite{GIL2} by using the $W$ conformal field theory. The corresponding tau function
was expressed as a~Fourier transform of the {\it semi-degenerate} $W_N$ conformal block
with the central charge $c=N-1$. Hence, by the uniqueness of a fundamental system of a Fuchsian system, $W$ conformal blocks are equal to 4d Nekrasov functions. We~note that in what follows, we assume convergence of~5d~Nekrasov partition functions, and moreover, Fourier transforms of them. Analyticity of~5d Nekrasov partition functions in cases different from our case
 was discussed in~\cite{FM}.

\subsection{Lax form}
Let $R$ be defined by
\begin{gather*}
R=\left\{n\in \Z^{N}\, \middle|\, \sum_{i=1}^N n^{(i)}=0\right\}\!,
\end{gather*}
which may be regarded as the root lattice of~$\frak{sl}_N$. We~set
\begin{gather*}
\tilde x=q^{N\sum_{i=1}^{m+1}\theta_i}x,\qquad
\tilde t_i=q^{N\sum_{j=i+1}^{m+1}\theta_j}t_i,\qquad i=1,\dots, m+1,\\
 s=\prod_{i=1}^m\prod_{j=1}^N s_{i,j},
\qquad s_i=\prod_{j=1}^Ns_{i,j},
\qquad i=1,\dots,m
\end{gather*}
for variables $x$, $t_i$, and $s_{i,j}$, $i=1,\dots, m$, $j=1,\dots,N$. We~denote $s_i^n=\prod_{j=1}^N s_{i,j}^{n^{(j)}}$ for $n\in R$
 and
$s^n=\prod_{i=1}^ms_i^{n_i}$ for $n=(n_1,\dots, n_m)\in R^m$.

We define a $N$ by $N$ matrix for $\theta_1,\dots,\theta_{m+1}\in\C$ and $\theta_0,\theta_\infty,\sigma_1,\dots,\sigma_{m+1}\in \C^N$ by
\begin{gather*}
Y^{(0,1)}_{i,j}(x,t)=\frac{1}{\tau_i^{(0,1)}(t)}
\\ \hphantom{Y^{(0,1)}_{i,j}(x,t)=}
{}\times \!\!\!\sum_{n=(n_1,\dots,n_m)\in R^m}\!\!\!\!\!s^n\F\!
\left(
{\phantom{00} \frac{1}{N}\phantom{000000000} (\theta_{\ell})_{\ell=1}^{m+1}
\atop
\theta_\infty-h_i \phantom{00} \theta_\infty-h_i+h_j\phantom{00} (\sigma_\ell+n_\ell)_{\ell=1}^m \phantom{00}
\theta_0}
 ;\,
 \tilde x, (\tilde t_{\ell})_{\ell=1}^{m+1}
 \right)\!.
 \end{gather*}
 Here, the function $\tau_i^{(0,1)}(t)$ is the normalization factor so that
\begin{gather*}
 Y^{(0,1)}(x,t)=\big( I+Y_1(t)x^{-1}+O(x^{-2})\big)
 \mathop{\rm diag}\big(x^{-\theta_\infty^{(1)}},\dots, x^{-\theta_\infty^{(N)}}\big),
 \end{gather*}
 where $I$ is the unit matrix of rank $N$.
 The matrix $Y^{(0,1)}(x,t)$ is an expansion around $x=\infty$. We~define the $N$ by $N$ matrices by
 \begin{gather*}
Y^{(k,k+1)}_{i,j}(x,t)=\frac{x^{-\sum_{\ell=1}^k\theta_\ell}}{\tau_i^{(k,k+1)}(t)}
\sum_{n\in R^m}s^ns_k^{h_j-h_N}
\\ \qquad
{}\times\F
\Bigg(
{\phantom{000} (\theta_{\ell})_{\ell=1}^{k}\phantom{00000000099} \frac{1}{N}\phantom{00} (\theta_\ell)_{\ell=k+1}^{m+1}\phantom{}
\atop
\theta_\infty-h_i \phantom{00} (\sigma_\ell-h_N+n_\ell)_{\ell=1}^{k}\phantom{00}
\sigma_k+h_j-h_N+n_k\phantom{00}(\sigma_{\ell}+n_{\ell})_{\ell=k+1}^{m} \phantom{00}
\theta_0}
 ;
 \\ \qquad\hphantom{\times\F\Bigg(0}
 (\tilde t_{\ell})_{\ell=1}^k,\tilde x,(\tilde t_\ell)_{\ell=k+1}^{m+1}
 \Bigg),
\end{gather*}
for $k=1,\dots, m$, and
\begin{gather*}
Y^{(m+1,m+2)}_{i,j}(x,t)
 \\ \qquad
{}=\frac{x^{-\sum_{\ell=1}^{m+1}\theta_\ell}}{\tau_i^{(m+1,m+2)}(t)}
\sum_{n\in R^m}s^n\F
\left(
{\phantom{000} (\theta_{\ell})_{\ell=1}^{m+1}\phantom{00000000099} \frac{1}{N}\phantom{}
\atop
\theta_\infty-h_i \phantom{00} (\sigma_\ell-h_N+n_\ell)_{\ell=1}^{m}\phantom{00}
\theta_0-h_j \phantom{00}
\theta_0}
 ;\,
 (\tilde t_{\ell})_{\ell=1}^{m+1},\tilde x
 \right)\!,
 \end{gather*}
 where the functions $\tau_i^{(k,k+1)}(t)$ are given recursively by
 \begin{gather}\label{eq_tau_{k,k+1}}
 \tau_i^{(k,k+1)}(t)=\tau_i^{(k-1,k)}(t)\,
 q^{N\theta_k^2/2+\theta_k/2+N\theta_k\sum_{i=1}^{k-1}\theta_i}t_k^{-\theta_k}.
 \end{gather}

By the definition of~$q$-conformal blocks,
asymptotic behaviours of the matrices of~$Y^{(0,1)}(x,t)$ and $Y^{(m+1,m+2)}(x,t)$ are given by
 \begin{gather}
 Y^{(0,1)}(x,t)=\widehat{Y}^{(0,1)}(x,t)
 \mathop{\rm diag}\big(x^{-\theta_\infty^{(1)}},\dots, x^{-\theta_\infty^{(N)}}\big),
\label{eq_aymptotic_behaviour_Y_infty}
\\
\widehat{Y}^{(0,1)}(x,t)= I+Y_1(t)x^{-1}+O\big(x^{-2}\big),\nonumber
 \\
 Y^{(m+1,m+2)}(x,t)=\widehat{Y}^{(m+1,m+2)}(x,t)
 \mathop{\rm diag}\big(x^{-\theta_0^{(1)}},\dots,x^{-\theta_0^{(N)}}\big)x^{-\sum_{i=1}^{m+1}\theta_i},
\label{eq_aymptotic_behaviour_Y_0}
\\
\widehat{Y}^{(m+1,m+2)}(x,t)=G(t)(I+O(x)), \nonumber
 \end{gather}
 where $I$ is the unit matrix of rank $N$, $Y_1(t)$ and $G(t)$ are $N$ by $N$ matrices. The off diagonal elements of~$Y_1(t)$ and all elements of the coefficient matrix $G(t)$ will be expressed as ratios of~tau functions.
Furthermore,
asymptotic behaviours of the matrices
$Y^{(k,k+1)}(x,t)$, $k=1,\dots,m$, are given by
\begin{gather}
Y^{(k,k+1)}(x,t)=\widehat{Y}^{(k,k+1)}(x,t)
\mathop{\rm diag}\big(x^{-\sigma_k^{(1)}-{1}/{N}}\!,\dots,
x^{-\sigma_k^{(N-1)}-{1}/{N}},x^{-\sigma_k^{(N)}+1-{1}/{N}}\big)x^{-\sum_{i=1}^k\theta_i},
\nonumber
\\
\widehat{Y}^{(k,k+1)}(x,t)=\sum_{n=-\infty}^{\infty}Y_n^{(k,k+1)}(t)x^n.
\label{eq_aymptotic_behaviour_Y_k}
\end{gather}

We assume that each $\widehat{Y}^{(k,k+1)}(x,t)$ is holomorphic in the domain
\begin{gather}\label{eq_assumption_hol_Y}
|t_1|>|t_2|>\cdots>|t_k|>\big| q^{1+N\sum_{i=1}^k\theta_i}x\big|,\qquad
\big|q^{N\sum_{i=1}^k\theta_i}x\big|>|t_{k+1}|>\cdots>|t_{m+1}|,
\end{gather}
which is a natural assumption from the point of view of conformal blocks.

\begin{Theorem}
Let
\begin{gather}
\label{eq_auumption_theta}
1<\big| q^{-N\theta_i}\big|<\frac{1}{|q|},\qquad i=1,\dots,m+1,
\\
\label{eq_assumption_poles_Y}
\big|q^{-N\sum_{j=i}^{k+1}\theta_j}t_{k+1}\big|<|t_i|,\qquad
i=1,\dots,m,\ k=i,\dots,m.
\end{gather}
Then, the matrices $Y^{(k,k+1)}(x,t)$, $k=0,1,\dots,m$, form a fundamental solution of the $q$-linear difference system
\begin{gather}\label{eq_multi_Lax}
T_{q,x}(Y(x,t))=A(x,t)Y(x,t),\quad\ T_{q,t_i}(Y(x,t))=B_i(x,t)Y(x,t),\quad\ i=1,\dots,m+1,\!\!
\end{gather}
with $N$ by $N$ matrices $A(x,t)$, $B_i(x,t)$ satisfying
\begin{gather}
A(x,t)=\frac{A_{m+1}x^{m+1}+\sum_{k=0}^{m}A_k(t)x^k}{\prod_{k=1}^{m+1}
\big(x-q^{-1-N\sum_{j=1}^{k-1}\theta_j}t_k\big)},
\label{eq_multi_A}
\qquad \det A(x,t)=\prod_{k=1}^{m+1}\frac{x-q^{-1-N\sum_{j=1}^{k}\theta_j}t_k}{x-q^{-1-N\sum_{j=1}^{k-1}\theta_j}t_k},
\\
A_{m+1}=\mathop{\rm diag}(q^{-\theta_\infty^{(1)}},\dots,q^{-\theta_\infty^{(N)}}), \nonumber
\\
A_0(t)=(-1)^{m+1}q^{-m-1-\sum_{i=1}^{m+1}\theta_i}\prod_{k=0}^m\big(q^{-N\sum_{j=1}^k\theta_j}t_{k+1}
\big)G(t)
\mathop{\rm diag}(q^{-\theta_0^{(1)}},\dots,q^{-\theta_0^{(N)}})
G(t)^{-1},\nonumber
\\
B_i(x,t)=\frac{xI+B_{i,0}(t)}{x-q^{-N\sum_{j=1}^{i}\theta_j}t_i},
\qquad
\det B_i(x,t)=\frac{x-q^{-N\sum_{j=1}^{i-1}\theta_j}t_i}{x-q^{-N\sum_{j=1}^{i}\theta_j}t_i}.
\label{eq_multi_B}
\end{gather}
\end{Theorem}
\begin{proof}
Thanks to Theorem~\ref{thm_CF_qCB}, we have for $k=0,1,\dots,m+1$
\begin{gather*}
Y^{(k,k+1)}_{i,j}(x,t)=\frac{x^{-\sum_{\ell=1}^k\theta_\ell}}{\tau_i^{(k,k+1)}(t)}
\sum_{n\in R^m}s^ns_k^{h_j-h_N}
\\ \qquad
{}\times\sum_{p=1}^N\F
\Bigg(
{\phantom{000} (\theta_{\ell})_{\ell=1}^{k+1}\phantom{00000000099} \frac{1}{N}\phantom{00} (\theta_\ell)_{\ell=k+2}^{m+1}\phantom{}
\atop
\theta_\infty-h_i \phantom{00} (\sigma_\ell-h_N+n_\ell)_{\ell=1}^{k}\phantom{00}
\sigma_{k+1}+n_{k+1}-h_p\phantom{00}(\sigma_{\ell}+n_{\ell})_{\ell=k+1}^{m} \phantom{00}
\theta_0}
;
\\ \qquad\hphantom{\times\sum_{p=1}^N\F\Bigg(0}
(\tilde t_{\ell})_{\ell=1}^{k+1},\tilde x,(\tilde t_\ell)_{\ell=k+2}^{m+1}
\Bigg)\B_{p,j}\left[\begin{matrix}
\theta_{k+1}&\frac{1}{N}
\\
\sigma_{k}-h_N+n_k&\sigma_{k+1}+n_{k+1}\end{matrix}
\Bigl| \frac{t_{k+1}}{q^{N\sum_{i=1}^{k+1}\theta_i}x}\right]
\\ \qquad
{}\times q^{N\theta_{k+1}^2/2-\theta_{k+1}/2}
\bigg( \frac{t_{k+1}}{q^{N\sum_{i=1}^{k+1}\theta_i}x}\bigg)^{\theta_{k+1}}.
\end{gather*}
By the periodicity~\eqref{eq_shifts_connection_matrix_theta+1},~\eqref{eq_shifts_conncetion_matrix_theta} of the connection matrix $\B$ and the definition~\eqref{eq_tau_{k,k+1}} of the function~$\tau_i^{(k,k+1)}(t)$, we obtain
\begin{gather*}
Y^{(k,k+1)}_{i,j}(x,t)=\frac{x^{-\sum_{\ell=1}^{k+1}\theta_\ell}}{\tau_i^{(k+1,k+2)}(t)}
\sum_{p=1}^N\B_{p,j}\left[\begin{matrix}
\theta_{k+1}&\frac{1}{N}
\\
\sigma_{k}-h_N&\sigma_{k+1}\end{matrix}
\Bigl| \frac{t_{k+1}}{q^{N\sum_{i=1}^{k+1}\theta_i}x}\right]\sum_{n\in R^m}s^ns_k^{h_j-h_N}
\\ \qquad
\times\F
\Bigg(
{\phantom{000} (\theta_{\ell})_{\ell=1}^{k+1}\phantom{00000000099} \frac{1}{N}\phantom{00} (\theta_\ell)_{\ell=k+2}^{m+1}\phantom{}
	\atop
	\theta_\infty-h_i \phantom{00} (\sigma_\ell-h_N+n_\ell)_{\ell=1}^{k}\phantom{00}
	\sigma_{k+1}+n_{k+1}-h_p\phantom{00}(\sigma_{\ell}+n_{\ell})_{\ell=k+1}^{m} \phantom{00}
	\theta_0}
;
\\ \qquad\hphantom{\times\F
\Bigg(0}
(\tilde t_{\ell})_{\ell=1}^{k+1},\tilde x,(\tilde t_\ell)_{\ell=k+2}^{m+1}
\Bigg).
\end{gather*}
 Hence, we have
 \begin{align*}
 Y^{(k,k+1)}_{i,j}(x,t)=s_k^{h_j-h_N}\sum_{p=1}^N
 Y^{(k+1,k+2)}_{i,p}(x,t)
 \B_{p,j}\left[\begin{matrix}
 \theta_{k+1}&\frac{1}{N}
 \\
 \sigma_{k}-h_N&\sigma_{k+1}\end{matrix}
 \Bigl| \frac{t_{k+1}}{q^{N\sum_{i=1}^{k+1}\theta_i}x}\right]\!.
 \end{align*}
Therefore we obtain for $k=0,1,\dots,m+1$
\begin{gather}\label{eq_CF_Ys}
Y^{(k,k+1)}(x,t)=Y^{(k+1,k+2)}(x,t)\B\left[\begin{matrix}
 \theta_{k+1}&\frac{1}{N}
 \\
 \sigma_{k}-h_N&\sigma_{k+1}\end{matrix}
\Bigl| \frac{t_{k+1}}{q^{N\sum_{i=1}^{k+1}\theta_i}x}\right]\nonumber
\\ \hphantom{Y^{(k,k+1)}(x,t)=}
{}\times\mathop{\rm diag}\big(s_k^{h_1-h_N},\dots,s_k^{h_{N-1}-h_N},1\big),
\end{gather}
where $\sigma_0=\theta_\infty+h_N$, $\sigma_{m+1}=\theta_0$,
$t_0=1$, and $s_0=1$.

Put
\begin{gather}
A(x,t)=T_{q,x}\big(Y^{(0,1)}(x,t)\big)Y^{(0,1)}(x,t)^{-1},\label{eq_def_A_B_1}
\\
B_i(x,t)=T_{q,t_i}\big(Y^{(0,1)}(x,t)\big)Y^{(0,1)}(x,t)^{-1},\qquad i=1,\dots,m+1.
\label{eq_def_A_B_2}
\end{gather}
Since all connection matrices are independent of
the $q$-shifts with respect to the variables $x$ and $t_i$ ($i=1,\dots,m+1$) by~\eqref{eq_q-shift_on_x_connection_matrix},
from~\eqref{eq_CF_Ys} the matrix functions $A(x,t)$, $B_i(x,t)$ of~$x$ are analytically continued on $\mathbb{P}^{1}$ as
\begin{gather*}
A(x,t)\!=T_{q,x}\big(Y^{(0,1)}(x,t)\big)Y^{(0,1)}(x,t)^{-1}
\!=\cdots\!=T_{q,x}\big(Y^{(m+1,m+2)}(x,t)\big)Y^{(m+1,m+2)}(x,t)^{-1},
\\
B_i(x,t)\!=T_{q,t_i}\big(Y^{(0,1)}(x,t)\big)Y^{(0,1)}(x,t)^{-1}
\!=\cdots\!=T_{q,t_i}\big(Y^{(m+1,m+2)}(x,t)\big)Y^{(m+1,m+2)}(x,t)^{-1}\!.
\end{gather*}
Therefore by the assumption on holomorphicity of~$\widehat{Y}^{k,k+1}(x,t)$, the functions $A(x,t)$, $B_i(x,t)$ of~$x$ are meromorphic functions of~$x$ on $\mathbb{P}^{1}$, namely, rational functions.

As a result, the matrices $Y^{(k,k+1)}(x,t)$ satisfy the $q$-linear difference system
\begin{gather*}
T_{q,x}(Y(x,t))=A(x,t)Y(x,t),\qquad
T_{q,t_i}(Y(x,t))=B_i(x,t)Y(x,t),\qquad i=1,\dots,m+1,
\end{gather*}
with $N\times N$ matrix rational functions $A(x,t)$, $B_i(x,t)$ of~$x$.

The forms of~$A(x,t)$ and $B_i(x,t)$ can be deduced from $\det Y^{(0,1)}(x,t)$.
The connection formula~\eqref{eq_CF_Ys} yields
\begin{gather}
\prod_{i=0}^m\frac{\big(q^{-N\sum_{j=1}^i\theta_j}t_{i+1}/x;q\big)_\infty}
{\big(q^{-N\sum_{j=1}^{i+1}\theta_j}t_{i+1}/x;q\big)_\infty}\,\det Y^{(0,1)}(x,t) \nonumber
\\ \qquad
{}=q^{N\sum_{i=1}^k\theta_i(N\sum_{i=1}^k\theta_i+1)/2}x^{N\sum_{i=1}^k\theta_i}
\prod_{i=0}^{k-1} s_i^{\sum_{j=1}^{N-1}(h_j-h_N)}t_{i+1}^{-N\theta_{i+1}}
\prod_{1\leq i<j\leq N}
\frac{\vartheta\big(\theta_\infty^{(i)}-\theta_\infty^{(j)}\big)}
{\vartheta\big(\theta_k^{(i)}-\theta_k^{(j)}\big)}\nonumber
\\ \qquad\phantom{=}
{}\times\prod_{i=k}^m\frac{\big(q^{-N\sum_{j=1}^i\theta_j}t_{i+1}/x;q\big)_\infty}
{\big(q^{-N\sum_{j=1}^{i+1}\theta_j}t_{i+1}/x;q\big)_\infty}\!
\prod_{i=0}^{k-1}\frac{\big(q^{1+N\sum_{j=1}^{i+1}\theta_j}x/t_{i+1};q\big)_\infty}
{\big(q^{1+N\sum_{j=1}^{i}\theta_j}x/t_{i+1};q\big)_\infty}
\, \det Y^{(k,k+1)}(x,t),\!\!\label{eq_det_Y_0_q}
\end{gather}
where $k=1,\dots,m+1$. The asymptotic behaviours~\eqref{eq_aymptotic_behaviour_Y_infty},
\eqref{eq_aymptotic_behaviour_Y_0},~\eqref{eq_aymptotic_behaviour_Y_k} imply that~\eqref{eq_det_Y_0_q} is a meromorphic function of~$x$ on $\mathbb{P}^1$.
Since $\widehat{Y}^{(k,k+1)}(x,t)$ is holomorphic in the domain
$\big|q^{-N\sum_{i=1}^k\theta_i}t_{k+1}\big|<|x|<\big|q^{-1-N\sum_{i=1}^{k-1}\theta_i}t_k\big|$ from the assumption~\eqref{eq_assumption_hol_Y}, the meromorphic function~\eqref{eq_det_Y_0_q} of~$x$ is holomorphic on $\mathbb{P}^1$ due to the condition~\eqref{eq_assumption_poles_Y}. Hence, it is a constant with respect to $x$, and therefore, using~\eqref{eq_aymptotic_behaviour_Y_infty} again, we obtain
\begin{gather}\label{eq_det_Y_infty}
\det Y^{(0,1)}(x,t)=\prod_{k=0}^m
\frac{\big(q^{-N\sum_{i=1}^{k+1}\theta_i}t_{k+1}/x;q\big)_\infty}
{\big(q^{-N\sum_{i=1}^k\theta_i}t_{k+1}/x;q\big)_\infty}.
\end{gather}

Then by the definitions~\eqref{eq_def_A_B_1},~\eqref{eq_def_A_B_2} we obtain
\begin{gather*}
\det A(x,t)=\prod_{k=1}^{m+1}\frac{x-q^{-1-N\sum_{i=1}^{k}\theta_i}t_k}{x-q^{-1-N\sum_{i=1}^{k-1}\theta_i}t_k},
\\
\det B_i(x,t)=\frac{x-q^{-N\sum_{j=1}^{i-1}\theta_j}t_i}{x-q^{-N\sum_{j=1}^{i}\theta_j}t_i},
\qquad i=1,\dots,m+1.
\end{gather*}
We recall that for each $k=0,1,\dots,m+1$, we have
\begin{gather*}
A(x,t)=T_{q,x}\big(Y^{(k,k+1)}(x,t)\big)Y^{(k,k+1)}(x,t)^{-1}.
\end{gather*}
from the definitions~\eqref{eq_def_A_B_1},~\eqref{eq_def_A_B_2} and the connection relations~\eqref{eq_CF_Ys}. Hence, in the domain
\begin{gather*}
\big| q^{-1-N\sum_{i=1}^k\theta_i}t_{k+1}\big|<|x|<
\big| q^{-1-N\sum_{i=1}^k\theta_i}t_k\big|,
\end{gather*}
 possible poles of~$A(x,t)$ are the zeros of~$\det\left( Y^{(k,k+1)}(x,t)\right)$. From~\eqref{eq_det_Y_0_q} we can read where
the zeros of~$\det\left( Y^{(k,k+1)}(x,t)\right)$ is.
 The conditions~\eqref{eq_auumption_theta} and~\eqref{eq_assumption_poles_Y} imply that $A(x,t)$ only has simple poles at $x=q^{-1-N\sum_{i=1}^{k-1}\theta_it_k}$, $k=1,\dots,m+1$.

Therefore, together with the asymptotic behaviour~\eqref{eq_aymptotic_behaviour_Y_infty}, we have
\begin{gather*}
A=\frac{A_{m+1}+\sum_{k=0}^{m}A_k(t)x^k}{\prod_{k=1}^{m+1}
\big(x-q^{-1-N\sum_{j=1}^{k-1}\theta_j}t_k\big)}.
\end{gather*}

The $q$-difference equation $Y(qx)=A(x)Y(x)$ entails that possible (simple) poles of the analytic continuations of~$Y^{(0,1)}(x,t)^{\pm 1}$, $Y^{(m+1,m+2)}(x,t)^{\pm 1}$ are
\begin{gather*}
Y^{(0,1)}(x,t)\colon \ q^{\ell-N\sum_{j=1}^{k-1}}t_k,\qquad
Y^{(0,1)}(x,t)^{-1}\colon \ q^{\ell-N\sum_{j=1}^{k-1}}t_k, \ q^{\ell-N\sum_{j=1}^k}t_k,
\\
Y^{(m+1,m+2)}(x,t)\colon \ q^{-\ell-1-N\sum_{j=1}^{k-1}}t_k, \ q^{-\ell-1-N\sum_{j=1}^k}t_k,
\\
Y^{(m+1,m+2)}(x,t)^{-1}\colon \ q^{-\ell-1-N\sum_{j=1}^{k-1}}t_k,
\end{gather*}
where $\ell \in \Z_{\geq 0}$, $k=1,\dots,m+1$. Hence, since
\begin{gather*}
B_i(x,t)=T_{q,t_i}\big(Y^{(0,1)}(x,t)\big) Y^{(0,1)}(x,t)^{-1}=
T_{q,t_i}\big(Y^{(m+1,m+2)}(x,t)\big) Y^{(m+1,m+2)}(x,t)^{-1},
\end{gather*}
$B_i(x,t)$ only has a simple pole at $x=q^{-N\sum_{j=1}^i\theta_j}t_i$. Therefore,
 we obtain~\eqref{eq_multi_B}.
\end{proof}

The compatibility conditions
\begin{gather}
\label{eq_com_pati_Lax}
T_{q,t_i}(A(x,t))B_i(x,t)=B_i(qx,t)A(x,t),\qquad i=1,\dots,m+1,
\end{gather}
of the Lax form~\eqref{eq_multi_Lax} with rational functions
$A(x,t)$ and $B_i(x,t)$ satisfying~\eqref{eq_multi_A} and~\eqref{eq_multi_B}
yield nonlinear $q$-difference system satisfied by the elements of~$A(x,t)$, $B_i(x,t)$
as functions of the variables $t$'s with the parameters $\theta$'s. When $m=1$ and $N=2$,
the derived nonlinear $q$-difference system is the $q$-difference VI equation~\cite{JS}. Fortunately,
the dependent variables of~$q$-$\mathrm{P_{VI}}$ are simply written by the elements
of~$A(x,t)$ and $B_2(x,t)$ and moreover they are expressed in terms of
Fourier transforms of~$q$-conformal blocks, namely, tau functions~\cite{Jimbo-Nagoya-Sakai}.
However, in the general case, nonlinear $q$-difference systems derived from compatibility conditions
are complicated and we have not found a way to nicely write down equation satisfied by the elements of~$A(x,t)$ and~$B_i(x,t)$. We~refer to~\cite{Sakai-q-Garnier} where the author
investigated the $N=2$ case and obtained a~$q$-analog of the Garnier system, see also
\cite{Nagao, Nagao-Yamada2, Nagao-Yamada1, Park}. We~also refer to~\cite{Fuji-Suzuki-RIMS2012} where the authors verified that the Lax form of a generalization $q$-$P_{(n,n)}$ of~$q$-$\mathrm{P_{VI}}$
obtained from the Drinfeld--Sokolov hierarchy is transformed to our Lax form~\eqref{eq_multi_Lax}
for $m=1$ by $q$-Laplace transformations \cite[Theorem~3.6]{Fuji-Suzuki-RIMS2012}.

The rational function $B_i(x,t)$ can be computed from the fundamental solution $Y^{(k,k+1)}(x,t)$.
In~particular, $B_{i,0}(t)$ has two expressions
\begin{gather}
B_{i,0}(t)=-q^{-N\sum_{j=1}^{i}\theta_j}t_iI+T_{q,t_i}(Y_1(t))-Y_1(t)
=-q^{-N\sum_{j=1}^{i}\theta_j}t_iT_{q,t_i}(G(t))G(t)^{-1},
\label{eq_Bi0}
\end{gather}
derived from
the asymptotic behaviours of the fundamental solution around $0$ and $\infty$. Recall that $Y_1(t)$, $G(t)$ are
defined by the asymtotoic expansions~\eqref{eq_aymptotic_behaviour_Y_infty} and~\eqref{eq_aymptotic_behaviour_Y_0}.
This produces
non-trivial relations among elements of~$Y_1(t)$ and $G(t)$, which can be viewed as
Fourier transforms of~$m+3$-point $q$-conformal blocks.

We define the tau function by
\begin{gather*}
\tau\left[ \theta | s,\sigma,t\right]=
q^{-N\Delta_{\theta_\infty}\sum_{k=1}^{m+1}\theta_k+N\sum_{k=1}^m\Delta_{N\theta_kh_1}
\sum_{j=k+1}^{m+1}\theta_j}
\\ \hphantom{\tau\left[ \theta | s,\sigma,t\right]=}
{}\times
\prod_{1\leq k<k'\leq N}G_q\big(1+\theta_\infty^{(k)}-\theta_\infty^{(k')}\big)
G_q\big(1-\theta_0^{(k)}+\theta_0^{(k')}\big)
\\ \hphantom{\tau\left[ \theta | s,\sigma,t\right]=}
{}\times\sum_{n\in R^m}s^n\F\left({
(\theta_p)_{p=1}^{m+1}
\atop
\theta_\infty\phantom{00}(\sigma_p+n_p)_{p=1}^m\phantom{00}\theta_0
};\,(\tilde t_p)_{p=1}^{m+1}
\right)\!.
\end{gather*}
Then
\begin{gather*}
\tau\left[ \theta | s,\sigma,t\right]=\sum_{n\in R^m}s^n \prod_{p=1}^{m+1}
t_p^{\Delta_{\sigma_{p-1}+n_{p-1}}-\Delta_{N\theta_ph_1}-\Delta_{\sigma_p+n_p}}
C[\theta| \sigma+n]Z[\theta|\sigma+n,t]
\end{gather*}
with the definition
\begin{gather*}
C[\theta|\sigma]
=\frac{\prod_{p=1}^{m+1}\prod_{k,k'=1}^NG_q\big(1+\sigma_{p-1}^{(k)}-\theta_p-\sigma_{p}^{(k')}\big)}
{\prod_{p=1}^{m}\prod_{k,k'=1}^NG_q\big(1+\sigma_p^{(k)}-\sigma_p^{(k')}\big)},
\\
Z[\theta|\sigma,t]
=\sum_{\bla_1,\dots, \bla_{m}\in\Y^N}
\prod_{p=1}^{m} \bigg( \frac{t_{p+1}}{t_{p}} \bigg)^{\left|\bla_p\right|}
 \frac{\prod_{p=1}^{m+1}\prod_{k,k'=1}^{N}N_{\lambda_{p-1}^{(k)}, \lambda_{p}^{(k')}}
\big(q^{\sigma_{p-1}^{(k)}-\theta_p-\sigma_{p}^{(k')}}\big)}
{\prod_{p=1}^{m}\prod_{k,k'=1}^{N} N_{\lambda_p^{(k)}, \lambda_{p}^{(k')}}
\big(q^{\sigma_p^{(k)}-\sigma_{p}^{(k')}}\big)}.
\end{gather*}
We note that the normalization factors $\tau_i^{(0,1)}(t)$, $i=1,\dots, N$, are
related to the tau function as
\begin{gather*}
\tau_i^{(0,1)}(t)=q^{N\Delta_{\theta_\infty}\sum_{k=1}^{m+1}\theta_k-N\sum_{k=1}^m\Delta_{N\theta_kh_1}
\sum_{j=k+1}^{m+1}\theta_j+\Delta_{\theta_\infty-h_i}}
\\ \hphantom{\tau_i^{(0,1)}(t)=}
{}\times
\bigg(\prod_{1\leq k<k'\leq N}G_q\big(1+\theta_\infty^{(k)}-\theta_\infty^{(k')}\big)
G_q\big(1-\theta_0^{(k)}+\theta_0^{(k')}\big)\bigg)^{-1}
\\ \hphantom{\tau_i^{(0,1)}(t)=}
{}\times \N'\Bigl({\kern-1pt\frac{1}{N}\atop{\theta_\infty-h_i \phantom{h_2} \theta_\infty}}
\Bigr)\tau\left[ \theta | s,\sigma,t\right].
\end{gather*}

Let us introduce other tau functions by
\begin{gather*}
\tau=\tau\left[ \theta | s,\sigma,t\right],\qquad
\tau_{ij}=\tau\left[\big(\theta_\infty-h_i+h_j, (\theta_p)_{p=1}^{m+2}\big)
|s,\sigma,t\right],
\\
\tilde\tau_{ij}=\tau\left[\big( \theta_\infty-h_i,(\theta_p)_{p=1}^{m+1},\theta_0-h_j\big)
|s,\sigma-h_N,t\right]
\end{gather*}
for $i,j=1,\dots,N$.

\begin{Proposition}
For $p=1,\dots,m+1$ and $i,j=1,\dots, N$, $i\neq j$, we have the determinantal identity\vspace{-1ex}
\begin{gather}\label{eq_det_tau_B}
T_{q,t_p}\left(\frac{\tau_{ij}}{\tau}\right)-\frac{\tau_{ij}}{\tau}=
\frac{q^{-N\sum_{k=1}^{p}\theta_k+\theta_p}t_pD_{ij}}{T_{q,t_p}(\tau)\tau^{N-1}}
\renewcommand*{\arraystretch}{.9}
 \left| \begin{matrix}
\tilde \tau_{11}&\tilde \tau_{12}&\cdots&\tilde \tau_{1N}
\\
\vdots&\vdots&&\vdots
\\
\tilde \tau_{j-1 1}&\tilde \tau_{j-1 2}&\cdots&
\tilde \tau_{j-1 N}
\\[.5ex]
T_{q,t_p}\left(\tilde \tau_{i1}\right)&T_{q,t_p}\left(\tilde\tau_{i2}\right)&\cdots&
T_{q,t_p}\left(\tilde \tau_{iN}\right)
\\[.5ex]
\tilde \tau_{j+1 1}&\tilde \tau_{j+1 2}&\cdots&
\tilde \tau_{j+1 N}
\\
\vdots&\vdots&&\vdots
\\
\tilde \tau_{N1}&\tilde \tau_{N2}&\cdots&\tilde \tau_{NN}
\end{matrix}\right|,
\end{gather}
where\vspace{-.5ex}
\begin{gather*}
D_{ij}=(-1)^{(m+1)(N-1)+1}
q^{N(N-1)\sum_{k=1}^{m+1}\theta_k/2+\theta_\infty^{(i)}-\theta_\infty^{(j)}+\sum_{k=1}^Nk(\theta_\infty^{(k)}-\theta_0^{(k)})}
\big[1-\theta_\infty^{(i)}+\theta_\infty^{(j)}\big]
\\ \hphantom{D_{ij}=}
{}\times
\prod_{k=1,\atop k\neq j}^N
\big[\theta_\infty^{(j)}-\theta_\infty^{(k)}\big]
\prod_{k<k'}\frac{1}{\big[\theta_\infty^{(k')}-\theta_\infty^{(k)}\big]
\big[\theta_0^{(k)}-\theta_0^{(k')}\big]}\, s\prod_{k=1}^ms_{k,N}^{-N}.
\end{gather*}
\end{Proposition}
\begin{proof}
By the relation~\eqref{eq_Bi0}, the $(i,j)$ entry of~$B_{p,0}(t)$ for $i\neq j$ satisfies
\begin{equation}\label{eq_prop_Bi0}
T_{q,t_p}\left(Y_1(t)_{ij}\right)-Y_1(t)_{ij}=-\frac{q^{-N\sum_{k=1}^{p}\theta_k}t_p}{\det G(t)}
\sum_{k=1}^NT_{q,t_p}\left(G(t)_{ik}\right)\widetilde{G(t)}_{jk},
\end{equation}
where $\widetilde{G(t)}$ is the cofactor matrix of~$G(t)$.

By the definition of~$Y_1(t)$ and $G(t)$ associated with $Y^{(0,1)}(x,t)$ and $Y^{(m+1,m+2)}(x,t)$,
we have for $i\neq j$
\begin{gather}
(Y_1(t))_{ij}=\frac{1}{\big[1-\theta_\infty^{(i)}+\theta_\infty^{(j)}\big]}
\prod_{k=1,\,k\neq i}^N\Gamma_q\big(\theta_\infty^{(i)}-\theta_\infty^{(k)}\big)
\prod_{k=1}^N\frac{1}{\Gamma_q\big(1+\theta_\infty^{(j)}-\theta_\infty^{(k)}\big)}
\,\frac{\tau_{ij}}{\tau},
\label{eq_Y1}
\\
G(t)_{ij}=q^{C_{ij}}\prod_{k=1}^{m+1}t_k^{\theta_k}
\prod_{k=1,\,k\neq i}^N\Gamma_q\big(\theta_\infty^{(i)}-\theta_\infty^{(k)}\big)
\prod_{k=1}^N\frac{1}{\Gamma_q\big(1+\theta_0^{(k)}-\theta_0^{(j)}\big)}
\, \frac{\tilde\tau_{ij}}{\tau},
\label{eq_G(t)}
\end{gather}
where\vspace{-.5ex}
\begin{gather*}
C_{ij}=\Delta_{\theta_0}-\Delta_{\theta_\infty}-\theta_0^{(j)}+\theta_\infty^{(i)}
-\frac{N}{2}\bigg(\sum_{k=1}^{m+1}\theta_k\bigg)^2
+\bigg(\frac{N}{2}-1-N\theta_0^{(j)}\bigg)\sum_{k=1}^{m+1}\theta_k.
\end{gather*}

Substituting~\eqref{eq_Y1} and~\eqref{eq_G(t)} into~\eqref{eq_prop_Bi0}
together with the expression
\begin{gather*}
\det G(t)^{-1}=(-1)^{(m+1)(N-1)}
q^{N\sum_{i=1}^{m+1}\theta_i(N\sum_{i=1}^{m+1}\theta_i+1)/2}
\\ \hphantom{\det G(t)^{-1}=}
{}\times\prod_{1\leq i<j\leq N}
\frac{\vartheta\big(\theta_\infty^{(i)}-\theta_\infty^{(j)}\big)}
{\vartheta\big(\theta_0^{(i)}-\theta_0^{(j)}\big)}
\prod_{k=0}^m s_k^{\sum_{i=1}^{N-1}(h_i-h_N)}t_{k+1}^{-N\theta_{k+1}}
\end{gather*}
of~$\det(G(t))$, we obtain the formula~\eqref{eq_det_tau_B}.
\end{proof}

\subsection{Schlesinger transformations}

Integer shifts on the characteristic exponents of Fuchsian systems preserving their monodromy are
called the Schlesinger transformations~\cite{Schlesinger}. In~\cite{Jimbo-Miwa-MPDII}, the authors
studied the Schlesinger trans\-for\-ma\-tions of monodromy preserving deformation including irregular
singular case and showed that the ratio $\tau'/\tau$ is a determinant of a matrix whose elements
are given in terms of the coefficients of fundamental solutions,
where $\tau$ is the Jimbo--Miwa--Ueno tau function
introduced in~\cite{Jimbo-Miwa-Ueno} and $\tau'$ is a Schlesinger transform of the original tau function. We~also mention a recent work~\cite{Ishikawa-Mano-Tsuda} where the authors presented
 another determinant formula for the ratio of~$\tau$-functions by using relations between Hermite's two approximation problems and particular Schlesinger transformation.

In~our case, the connection matrix
has periodicities~\eqref{eq_shifts_conncetion_matrix_sigma} and~\eqref{eq_shifts_conncetion_matrix_theta} with shifts of the parameters. This fact implies that
the following shifts of the parameters
\begin{gather*}
\theta_\infty\to \theta_\infty-h_i,\qquad
\theta_0\to \theta_0-h_j,\qquad
1\leq i,j\leq n,
\\
\theta_i\to\theta_i+\frac{1}{N},\qquad 1\leq i\leq m+1,
\end{gather*}
can be regarded as $q$-Schlesinger transformations and
yield $q$-difference
linear equations with rational coefficients satisfied
by $Y^{(k,k+1)}(x,t)$, $k=0,1,\dots,m+1$.

Let transformations $r_i$, $i=1,\dots,m+1$, and $p$
on the parameters be defined as follows.
\begin{gather*}
r_i(\theta_\infty)=\theta_\infty-h_1,\quad
r_i(\sigma_j)=\sigma_j-h_1,\quad j=1,\dots,i-1,\quad\ \
r_i(\sigma_j)=\sigma_j,\quad j=i,\dots,m,
\\
r_i(\theta_0)=\theta_0,\qquad
r_i(\theta_j)=\theta_j+\frac{1}{N}\delta_{i,j},\qquad
j=1,\dots,m+1,
\\
r_i(t_j)=t_j,\qquad j=1,\dots,i,\qquad
r_i(t_j)=qt_j,\qquad j=i+1,\dots,m+1,
\\
p(\theta_\infty)=\theta_\infty-h_1,\qquad
p(\sigma_j)=\sigma_j-h_1,\qquad j=1,\dots,m,\qquad
p(\theta_0)=\theta_0-h_1.
\end{gather*}
We show below that from the $q$-Schlesinger transformations $r_i$ and $p$, we obtain
determinantal identities for the tau functions.

\begin{Proposition}
The functions
\begin{gather*}
(-1)^{N\sum_{j=1}^k\theta_j}Y^{(k,k+1)}(x,t)x^{\theta_\infty^{(N)}},\qquad k=0,1,\dots, m+1,
\end{gather*}
satisfy the $q$-difference linear equation
\begin{gather*}
r_i(Y(x,t))=R_i(x,t)Y(x,t)
\end{gather*}
with
\begin{gather}
\det R_i(x,t)=x-q^{-N\sum_{j=1}^i\theta_j}t_i,\nonumber
\quad
R_i(x,t)=R_0x+R_{i,0}(t),\quad
R_0=\mathop{\rm diag}(1,0,\dots,0),\nonumber
\\
R_{i,0}(x,t)=r_i(Y_1(t))R_0+I-R_0-R_0Y_1(t)
=-r_i(G(t))G(t)^{-1}. \label{eq_R_{i,0}}
\end{gather}
\end{Proposition}
\begin{proof}
Put
\begin{gather*}
R_i(x,t)=r_i\big(Y^{(0,1)}(x,t)\big)Y^{(0,1)}(x,t)^{-1}x^{1/N}.
\end{gather*}
Then we have
\begin{gather*}
R_i(x,t)=\begin{cases} r_i\big(Y^{(k,k+1)}(x,t)\big)Y^{(k,k+1)}(x,t)^{-1}x^{1/N}, & k=1,\dots,i-1,
\\ -r_i\big(Y^{(k,k+1)}(x,t)\big)Y^{(k,k+1)}(x,t)^{-1}x^{1/N}, & k=i,\dots,m+1,
\end{cases}
\end{gather*}
due to the connection relations~\eqref{eq_CF_Ys} and the periodicities~\eqref{eq_shifts_conncetion_matrix_sigma},
\eqref{eq_shifts_conncetion_matrix_theta} of the connection matrix.

{\sloppy
By the definition of~$q$-conformal blocks,
$R_i(x,t)$ is a meromorphic function on $\mathbb{P}^1$,
and hence it is a rational function of~$x$.
The determinant of~$R_i(x,t)$ is computed from
$\det Y^{(0,1)}(x,t)$~\eqref{eq_det_Y_infty}. Then,
the form of~$R_i(x,t)$ is obtained from the asymptotic
expansions of~$Y^{(0,1)}(x,t)$ and \linebreak $Y^{(m+1,m+2)}(x,t)$.
}
\end{proof}

Since we can express $R_{i,0}(x,t)$ in two different way~\eqref{eq_R_{i,0}}, we obtain
the following theorem.
\begin{Theorem}\label{thm det tau r}
Let $1\le i\leq m+1$.
\begin{enumerate}\itemsep=0pt
\item[$1.$] For $2\leq a,b\leq N$,
we have\vspace{-1ex}
\begin{gather*}
t_i^{1/N}q^{\theta_\infty^{(1)}-\sum_{j=1}^i\theta_j}
\frac{D_{ab}}{r_i(\tau)\tau^{N-1}}
\left| \begin{matrix}
\tilde \tau_{11}&\tilde \tau_{12}&\cdots&\tilde \tau_{1N}
\\
\vdots&\vdots&&\vdots
\\
\tilde \tau_{b-1 1}&\tilde \tau_{b-1 2}&\cdots&
\tilde \tau_{b-1 N}
\\[.5ex]
q^{-\theta_0^{(1)}}r_i\left(\tilde \tau_{a1}\right)
&q^{-\theta_0^{(2)}}r_i\left(\tilde\tau_{a2}\right)&\cdots&
q^{-\theta_0^{(N)}}r_i\left(\tilde \tau_{aN}\right)
\\
\tilde \tau_{b+1 1}&\tilde \tau_{b+1 2}&\cdots&
\tilde \tau_{b+1 N}
\\
\vdots&\vdots&&\vdots
\\
\tilde \tau_{N1}&\tilde \tau_{N2}&\cdots&\tilde \tau_{NN}
\end{matrix}\right|
\\[1ex] \qquad
{}=-\delta_{ab}\frac{\prod_{j=1}^N\Gamma_q\big(1+\theta_\infty^{(b)}-\theta_\infty^{(j)}\big)}
{\prod_{j=1,\,j\neq a}^N\Gamma_q\big(\theta_\infty^{(a)}-\theta_\infty^{(j)}\big)}
\frac{\big[1-\theta_\infty^{(a)}+\theta_\infty^{(b)}\big]}
{\big[\theta_\infty^{(a)}-\theta_\infty^{(1)}\big]}.
\end{gather*}

\item[$2.$] For $2\leq b\leq N$, we have\vspace{-1ex}
\begin{gather*}
t_i^{1/N}q^{\theta_\infty^{(1)}-1-\sum_{j=1}^i\theta_j}
\frac{D_{1b}}{r_i(\tau)\tau^{N-1}}
\left| \begin{matrix}
\tilde \tau_{11}&\tilde \tau_{12}&\cdots&\tilde \tau_{1N}
\\
\vdots&\vdots&&\vdots
\\
\tilde \tau_{b-1 1}&\tilde \tau_{b-1 2}&\cdots&
\tilde \tau_{b-1 N}
\\[.5ex]
q^{-\theta_0^{(1)}}r_i\left(\tilde \tau_{11}\right)
&q^{-\theta_0^{(2)}}r_i\left(\tilde\tau_{12}\right)&\cdots&
q^{-\theta_0^{(N)}}r_i\left(\tilde \tau_{1N}\right)
\\
\tilde \tau_{b+1 1}&\tilde \tau_{b+1 2}&\cdots&
\tilde \tau_{b+1 N}
\\
\vdots&\vdots&&\vdots
\\
\tilde \tau_{N1}&\tilde \tau_{N2}&\cdots&\tilde \tau_{NN}
\end{matrix}\right|
\\ \qquad
{}=\frac{\tau_{1b}}{\tau}\prod_{j=2}^N\big[\theta_\infty^{(1)}-\theta_\infty^{(j)}-1\big].
\end{gather*}

\item[$3.$] For $2\leq a\leq N$, we have
\begin{gather*}
t_i^{1/N}q^{\theta_\infty^{(1)}-\sum_{j=1}^i\theta_j}
\frac{D_{a1}}{r_i(\tau)\tau^{N-1}}
\left| \begin{matrix}
q^{-\theta_0^{(1)}}r_i\left(\tilde \tau_{a1}\right)
&q^{-\theta_0^{(2)}}r_i\left(\tilde\tau_{a2}\right)&\cdots&
q^{-\theta_0^{(N)}}r_i\left(\tilde \tau_{aN}\right)
\\
\tilde \tau_{2 1}&\tilde \tau_{2 2}&\cdots&
\tilde \tau_{2 N}
\\
\vdots&\vdots&&\vdots
\\
\tilde \tau_{N1}&\tilde \tau_{N2}&\cdots&\tilde \tau_{NN}
\end{matrix}\right|
\\ \qquad
{}=-r_i\bigg(\frac{\tau_{a1}}{\tau}\bigg)
\big[1-\theta_\infty^{(a)}+\theta_\infty^{(1)}\big]
\prod_{j=2,\,j\neq a}^N\big[\theta_\infty^{(1)}-\theta_\infty^{(j)}\big].
\end{gather*}
\end{enumerate}
\end{Theorem}

Similarly to the $q$-Schlesinger transformation $r_i$, we have the following proposition and theorem
on the $q$-Schlesinger transformation $p$.
\begin{Proposition}
The functions $Y^{(k,k+1)}(x,t)$, $k=0,1,\dots,m+1$,
satisfy the $q$-difference linear equation
\begin{gather*}
p\left(Y(x,t)\right)=P(x,t)Y(x,t)x^{-1/N}
\end{gather*}
with
\begin{gather*}
\det P(x,t)=x,\qquad
P(x,t)=P_0x+P_{i,0}(t),\qquad
P_0=\mathop{\rm diag}(1,0,\dots,0),
\\
P_{i,0}(x,t)=p(Y_1(t))P_0+I-P_0-P_0Y_1(t)=p(G(t))G(t)^{-1}.
\end{gather*}
\end{Proposition}

\begin{Theorem}\quad
\begin{enumerate}\itemsep=0pt
\item[$1.$] For $2\leq a,b\leq N$,
we have
\begin{gather*}
q^{\theta_\infty^{(1)}-\theta_0^{(1)}-\sum_{j=1}^N\theta_j}
\frac{D_{ab}}{p(\tau)\tau^{N-1}}
\left| \begin{matrix}
\tilde \tau_{11}&\tilde \tau_{12}&\cdots&\tilde \tau_{1N}
\\
\vdots&\vdots&&\vdots
\\
\tilde \tau_{b-1 1}&\tilde \tau_{b-1 2}&\cdots&
\tilde \tau_{b-1 N}
\\[.5ex]
0
&\big[\theta_0^{(1)}-\theta_0^{(2)}\big]p\big(\tilde\tau_{a2}\big)&\cdots&
\big[\theta_0^{(1)}-\theta_0^{(N)}\big]p\big(\tilde \tau_{aN}\big)
\\
\tilde \tau_{b+1 1}&\tilde \tau_{b+1 2}&\cdots&
\tilde \tau_{b+1 N}
\\
\vdots&\vdots&&\vdots
\\
\tilde \tau_{N1}&\tilde \tau_{N2}&\cdots&\tilde \tau_{NN}
\end{matrix}\right|
\\ \qquad
{}=\delta_{ab}\frac{\prod_{j=1}^N\Gamma_q\big(1+\theta_\infty^{(b)}-\theta_\infty^{(j)}\big)}
{\prod_{j=1,\,j\neq a}^N\Gamma_q\big(\theta_\infty^{(a)}-\theta_\infty^{(j)}\big)} \frac{\big[1-\theta_\infty^{(a)}+\theta_\infty^{(b)}\big]} {\big[\theta_\infty^{(a)}-\theta_\infty^{(1)}\big]}.
\end{gather*}

\item[$2.$] For $2\leq b\leq N$, we have
\begin{gather*}
q^{\theta_\infty^{(1)}-\theta_0^{(1)}-1-\sum_{j=1}^N
\theta_j}
\frac{D_{1b}}{p(\tau)\tau^{N-1}}
\left| \begin{matrix}
\tilde \tau_{11}&\tilde \tau_{12}&\cdots&\tilde \tau_{1N}
\\
\vdots&\vdots&&\vdots
\\
\tilde \tau_{b-1 1}&\tilde \tau_{b-1 2}&\cdots&
\tilde \tau_{b-1 N}
\\
0
&\big[\theta_0^{(1)}-\theta_0^{(2)}\big]p\left(\tilde\tau_{12}\right)&\cdots&
\big[\theta_0^{(1)}-\theta_0^{(N)}\big]p\left(\tilde \tau_{1N}\right)
\\
\tilde \tau_{b+1 1}&\tilde \tau_{b+1 2}&\cdots&
\tilde \tau_{b+1 N}
\\
\vdots&\vdots&&\vdots
\\
\tilde \tau_{N1}&\tilde \tau_{N2}&\cdots&\tilde \tau_{NN}
\end{matrix}\right|
\\ \qquad
{}=-\frac{\tau_{1b}}{\tau}\prod_{j=2}^N\big[\theta_\infty^{(1)}-\theta_\infty^{(j)}-1\big].
\end{gather*}

\item[$3.$] For $2\leq a\leq N$, we have\vspace{-1ex}
\begin{gather*}
q^{\theta_\infty^{(1)}-\theta_0^{(1)}-\sum_{j=1}^N
\theta_j}
\frac{D_{a1}}{p(\tau)\tau^{N-1}}\renewcommand*{\arraystretch}{0.9}
\left| \begin{matrix}
0
&\big[\theta_0^{(1)}-\theta_0^{(2)}\big]p\left(\tilde\tau_{a2}\right)&\cdots&
\big[\theta_0^{(1)}-\theta_0^{(N)}\big]p\left(\tilde \tau_{aN}\right)
\\
\tilde \tau_{2 1}&\tilde \tau_{2 2}&\cdots&\tilde \tau_{2 N}
\\
\vdots&\vdots&&\vdots
\\
\tilde \tau_{N1}&\tilde \tau_{N2}&\cdots&\tilde \tau_{NN}
\end{matrix}\right|
\\ \qquad
{}=p\bigg(\frac{\tau_{a1}}{\tau}\bigg)
\big[1-\theta_\infty^{(a)}+\theta_\infty^{(1)}\big]
\prod_{j=2,\,j\neq a}^N\big[\theta_\infty^{(1)}-\theta_\infty^{(j)}\big].
\end{gather*}
\end{enumerate}
\end{Theorem}

\begin{Remark}
In~the case of~$N=2$, the determinantal identities above become bilinear equations
satisfied by the tau functions. On the other hand, in the previous work~\cite{Jimbo-Nagoya-Sakai},
we gave conjectural bilinear equations of the tau functions for $N=2$ and $m=1$.
They are not equal to, and it is not clear whether the bilinear equations obtained in this paper
recover the conjectural bilinear equations (equations~(3.16)--(3.23) in~\cite{Jimbo-Nagoya-Sakai}).\vspace{-1ex}
\end{Remark}

\begin{Remark}
When $N=2$ and $m>1$, the compatibility condition~\eqref{eq_com_pati_Lax}
of the Lax pair yields the $q$-Garnier
system. For its $q\to 1$ limit, namely, the Garnier system, bilinear equations
of the tau functions were presented in~\cite{Suzuki-Garnier}. It is easily seen from
the transformation of the parameters by $r_i$ that the bilinear equations in Theorem~\ref{thm det tau r}(2),~(3) are $q$-analogs of (part of) the Hirota--Miwa equations (equations~(3.47) in~\cite{Suzuki-Garnier}).\vspace{-1ex}
\end{Remark}

\subsection*{Acknowledgements}
This work is partially supported by JSPS KAKENHI Grant Number JP18K03326.
The author thanks the referees for valuable suggestions and comments.


\pdfbookmark[1]{References}{ref}
\LastPageEnding

\end{document}